\documentclass[final,1p,times]{elsarticle}

\usepackage{titlesec}
\usepackage{latexsym}
\usepackage{amsfonts}
\usepackage{amsthm}
\usepackage{amsmath}
\usepackage{amssymb}
\usepackage{color}
\usepackage{mathrsfs}
\usepackage[normalem]{ulem}
\usepackage[toc,page]{appendix}

\definecolor{light}{gray}{.96}

\newcommand\think[1]{}
\newcommand{\curl}[1]{\mathrm{Curl}(#1)}

\numberwithin{equation}{section}

\def\v{\mathrm{v}}

\newcommand\R{\mathbb{R}}

\newcommand\1{u_1}
\newcommand\2{u_2}

\newcommand\dom{\mathcal{O}}

\def\T{\mathbb{T}^2}

\newtheorem{theorem}{Theorem}[section]

\newtheorem{alg}[theorem]{Algorithm}

\newtheorem{corollary}[theorem]{Corollary}

\newtheorem{defn}[theorem]{Definition}

\newtheorem{example1}[theorem]{Example}
\newtheorem{exercise1}[theorem]{Exercise}
\newtheorem{lemma}[theorem]{Lemma}

\newtheorem{proposition}[theorem]{Proposition}
\newtheorem{remark1}[theorem]{Remark}

\newtheorem{digression1}[theorem]{Digression}

\newenvironment{definition}{
\begin{defn}
	\normalfont}{
\end{defn}
}

\newenvironment{remark}{
\begin{remark1}
	\normalfont}{
\end{remark1}
}

\newtheoremstyle{AppALem}{1}{1}
  {\itshape}{0pt}{\bfseries}{.}{ }
   {\thmname{Lemma }\thmnumber{A.{#2}}{\thmnote{}}}
   \theoremstyle{AppALem}\newtheorem{lemmaA}{Lemma}

\newtheoremstyle{AppBLem}{1}{1}
  {\itshape}{0pt}{\bfseries}{.}{ }
   {\thmname{Lemma }\thmnumber{B.{#2}}{\thmnote{}}}
   \theoremstyle{AppBLem}

\newtheoremstyle{AppBRem}{1}{1}
  {\itshape}{0pt}{\bfseries}{.}{ }
   {\thmname{Remark }\thmnumber{B.{#2}}{\thmnote{}}}
   \theoremstyle{AppBRem}\newtheorem{remarkB}{Remark}

\newtheoremstyle{AppBCor}{1}{1}
  {\itshape}{0pt}{\bfseries}{.}{ }
   {\thmname{Corollary }\thmnumber{B.{#2}}{\thmnote{}}}
   \theoremstyle{AppBCor}

\definecolor{darkred}{rgb}{0.9,0.1,0.1}

\definecolor{darkblue}{rgb}{0.1,0.1,0.9}

\begin{document}
\begin{frontmatter}
\title{2D Constrained Navier-Stokes Equations}
\author[ZB]{Zdzis\l aw Brze\'{z}niak}
\ead{zdzislaw.brzezniak@york.ac.uk}
\author[ZB]{Gaurav Dhariwal\corref{cor1}}
\ead{gd673@york.ac.uk}
\author[MM]{Mauro Mariani}
\ead{mariani@mat.uniroma1.it}
\address[ZB]{Department of Mathematics, University of York, Heslington, York,
YO10 5DD, UK}
\address[MM]{Dipartimento di Matematica
Universit\`a degli studi di Roma La Sapienza, Piazzale Aldo Moro, 5, Roma, 00185, Italy}
\tnotetext[t1]{The research of Gaurav Dhariwal is supported by Department of Mathematics, University of York. The research of Mauro Mariani has been 
 supported by  the A*MIDEX
project (n.\ ANR11IDEX000102) funded by the \emph{Investissements
d'Avenir} French Government program, managed by the French National
Research Agency (ANR). The research of Zdzislaw Brze{\'z}niak has been partially supported by the Leverhulme project grant ref no RPG-2012-514.
}

\cortext[cor1]{Corresponding author.}

\begin{abstract}
We study 2D Navier-Stokes equations with a constraint on $L^2$ energy of the solution. We prove the existence and uniqueness of a global solution for the constrained Navier-Stokes equation on $\R^2$ and $\T$, by a fixed point argument. We also show that the solution of constrained Navier-Stokes converges to the solution  of Euler equation as viscosity $\nu$ vanishes.
\end{abstract}


\end{frontmatter}

\section{Introduction}
\label{s:1}

The motivation for this paper is twofold. Firstly Caglioti \textit{et.al.} in \cite{CPR09} studied the well-posedness and asymptotic behaviour of two dimensional Navier-Stokes equations in the vorticity form with two constraints: constant energy $E(\omega)$ and moment of inertia $I(\omega)$
\[
\dfrac{\partial \omega}{\partial t} + u \cdot \nabla \omega = \nu \Delta \omega - \nu\, \rm{div}\left[ \omega \nabla \left(b \psi +  a \dfrac{|x|^2}{2} \right) \right],
\]
which can be rewritten as
\begin{equation}
\label{eq:cpr}
\dfrac{\partial \omega}{\partial t} + u \cdot \nabla \omega = \nu\, \rm{div}\left[ \omega \nabla \left(\log{\omega} -b \psi - a \dfrac{|x|^2}{2} \right) \right],
\end{equation}
where $\omega = \curl{u}$, $a = a(\omega)$ and $b = b(\omega)$ are the Lagrange multipliers associated to those constraints and
\[E(\omega) = \int_{\R^2} \psi \omega\,dx, \quad I(\omega) = \int_{\R^2} |x|^2 \omega \,dx, \quad \psi = - \Delta^{-1} \omega. \] They were able to show the existence of a unique classical global-in-time solution to \eqref{eq:cpr} for a family of initial data \cite[Theorem~5]{CPR09}. They were also able to prove that the solution to \eqref{eq:cpr} converges, as time tends to $+ \infty$, to the unique solution of an associated microcanonical variational problem \cite[Theorem~8]{CPR09}.

Secondly, Rybka \cite{[Rybka06]} and Caffarelli \& Lin \cite{[CL09]} study the linear heat equation with constraints. Rybka studied heat flow on a manifold $\mathcal{M}$ given by
\[\mathcal{M} = \left\{u \in L^2(\Omega) \cap C(\Omega) : \int_{\Omega} u^k(x)\,dx = C_k,\, k = 1, \dots, N \right\},\]
where $\Omega$ denotes a connected bounded region in $\R^2$ with smooth boundary. He proved \cite[Theorem~2.5]{[Rybka06]} the existence of the unique global solution for the projected heat equation
\begin{equation}
\label{eq:rybka}
\begin{cases}
\frac{du}{dt} = \Delta u - \sum_{k = 1}^N \lambda_k u^{k-1} \quad \mbox{in}\,\, \Omega \subset \R^2,\\
\frac{\partial u}{\partial n} = 0 \,\, \mbox{on}\,\, \partial \Omega, \quad \quad u(0,x) = u_0,
\end{cases}
\end{equation}
where $\lambda_k = \lambda_k(u)$ are such that $u_t$ is orthogonal to $\mathrm{Span}\left\{u^{k-1}\right\}$. He also showed that the solutions to \eqref{eq:rybka} converges to a steady state as time tends to $+ \infty$.\\
On the other hand Caffarelli and Lin initially establish the existence and uniqueness of a global, energy-conserving solution to the heat equation \cite[Theorem~1.1]{[CL09]}. They were then able to extend these results to more general family of singularly perturbed systems of nonlocal parabolic equations \cite[Theorem~3.1]{[CL09]}. Their main result was to prove the strong convergence of the solutions to these perturbed systems to some weak-solutions of the limiting constrained nonlocal heat flows of maps into a singular space.

In this paper we consider a problem which links the aforementioned works. We consider Navier-Stokes equations as in \cite{CPR09}, but subject to the same energy constraint as in \cite{[CL09],[Rybka06]}. Contrary to \cite{CPR09} we prove global-in-time existence of the solution but only on a torus, namely in the periodic case. Surprisingly our proof of global existence does not hold for a general bounded domain, although the local existence holds. We also prove our result of global existence of the solution for $\R^2$.
We additionally show that, in vanishing viscosity limit, the solution of the constrained equation \eqref{eq:1.1} below, converges to the Bardos solution (see \cite{[Bardos]}) of the Euler equation (formally obtained setting $\nu = 0$).

We are interested in the Cauchy problem
\begin{align}
\label{eq:1.1}
\begin{cases}
\dfrac{du}{dt} = - \nu Au + \nu|\nabla u|^2\,u - B(u,u),\\
u(0) = u_0,
\end{cases}
\end{align}
where $u \in \rm{H}$, and $\rm{H}$ is a space of divergence free, mean zero vector fields on a torus, see \eqref{eq:functional on torus} below for a precise definition.

The above problem has a local \emph{maximal} solution for each $u_0 \in \rm{V} \cap \mathcal{M}$, where $\rm{V}$ is defined in \eqref{eq:functional on torus} and
\[\mathcal{M} = \{ u \in \rm{H} : |u| = 1\}.\]
Moreover $u(t) \in \mathcal{M}$ for all times $t$. This result is true both for NSEs on a bounded domain or with periodic boundary conditions (i.e. on a torus). In a more geometrical fashion, equation (\ref{eq:1.1}) can be also written as
\[\dfrac{du}{dt} = - \nabla_{\mathcal{M}} \mathcal{E}(u) - B(u,u),\]
where $\mathcal{E}(u) = \frac12 |\nabla u|^2, u \in \mathcal{M}$ and $\nabla_{\mathcal{M}}\mathcal{E}(u)$ is the gradient of $\mathcal{E}$ with respect to $\rm{H}$-norm projected onto $T_u\mathcal{M}$. The remarkable feature of this is that on a torus $\nabla_{\mathcal{M}}\mathcal{E}(u)$ and $B(u,u)$ are orthogonal in $\rm{H}$. This orthogonality holds for the Navier-Stokes without constraint too, i.e. on a torus $\nabla \mathcal{E}(u)$ is orthogonal to $B(u,u)$ in $\rm{H}$. The fact that this constraint preserves the orthogonality somehow makes it a natural constraint.

Hence in at least heuristic way
\begin{align*}
\frac{d}{dt}\mathcal{E}(u(t)) & = \left\langle \nabla_{\mathcal{M}}\mathcal{E}(u(t)), \frac{du}{dt} \right\rangle_{\rm{H}} \\
& = \langle \nabla_{\mathcal{M}}\mathcal{E}(u(t)), - \nabla_{\mathcal{M}}\mathcal{E}(u(t)) - B(u,u) \rangle_{\rm{H}}\\
& = - |\nabla_{\mathcal{M}}\mathcal{E}(u(t))|^2,
\end{align*}
so that $\mathcal{E}(u(t))$ is decreasing and thus the $H^{1,2}$ norm of the solution remains bounded.

Next we state the two main results of the paper on a torus.

\begin{theorem}
\label{thm1.1}
Let $u_0 \in \rm{V} \cap \mathcal{M}$ and $X_T = C([0,T]; \rm{V}) \cap L^2(0,T; \rm{E})$. Then for every $\nu > 0$ there exists a global and locally unique solution $u \in X_T$ of \eqref{eq:1.1}.
\end{theorem}
\medskip
The space $X_T$ with more details and the precise definition of the solution of \eqref{eq:1.1} will be given in the Section~\ref{s:3}. Theorem~\ref{thm1.1} will be proved in steps in Sections~\ref{s:3} and \ref{s:4}.

\begin{theorem}
\label{thm1.2}
Let $u_0, u_0^\nu \in \rm{V} \cap \mathcal{M}$ and $u^\nu$ be the solution of \eqref{eq:1.1} (existence and uniqueness of $u^\nu$ follows from Theorem~\ref{thm1.1}). Assume that $u_0^\nu \to u_0$ in $\rm{V}$ as $\nu \downarrow 0$, and that $\mathrm{Curl}(u_0^\nu)$ stays uniformly bounded in $L^\infty(\T)$. Then for each $T>0$, $u^\nu$ converges in $C([0,T];L^2(\T))$ to the unique solution $u$ of the limiting equation (namely \eqref{eq:1.1} with $\nu=0$).
\end{theorem}
\medskip

We end the introduction with a brief description of the content of the paper. In Section~\ref{s:2}, we introduce a constrained Navier-Stokes equation. In Section~\ref{s:3}, a precise definition of the solution is given, and local existence and uniqueness are proved, together with some basic properties of the solution. In Section~\ref{s:4}, global existence is established. Finally, in Section~\ref{s:5} we prove Theorem~\ref{thm1.2}.

\section{Constrained Navier-Stokes equations}
\label{s:2}

\subsection{General Notations}
\label{s:2.1}
Let $\mathcal{O}$ be either a bounded domain in $\R^2$, $\R^2$ or $\T$.  For $p \in [1, \infty]$ and $k \in \mathbb{N}$, the Lebesgue and Sobolev spaces of $\R^2$-valued functions will be denoted by $L^p(\mathcal{O}, \R^2)$ and $W^{k,p}(\mathcal{O}, \R^2)$ respectively, and often $L^p$ and $W^{k,p}$ whenever the context is understood. The usual scalar product on $L^2$ is denoted by $\langle u,\v \rangle$ for $u, \v \in L^2$. The associated norm is given by $|u|, u \in L^2$. We also write $W^{k,2}(\mathcal{O}, \R^2) := H^k$ and will denote it's norm by $\|\cdot \|_{H^k}$. In particular the scalar product for $H^1$ is given by
\[\langle u , \v \rangle_{H^1} = \langle u , \v \rangle + \langle \nabla u, \nabla \v \rangle,\;\; u, \v \in H^1,\]
and thus the norm is
\[\|u\|_{H^1} = \left[ |u|^2 + |\nabla u|^2\right]^{1/2}. \]
In the following two subsections we will introduce some additional spaces. The structure of the spaces will depend on the choice of $\mathcal{O}$.

\subsection{Functional setting for $\R^2$}
\label{s:2.2}
We consider the whole space $\R^2$.  We introduce the following spaces:

\begin{equation}
\label{eq:functional}
\begin{split}
\rm{H} & = \{u \in L^2(\R^2, \R^2) : \nabla \cdot u = 0\},\\
\rm{V} & = H^1 \cap \rm{H}.
\end{split}
\end{equation}

We endow ${\rm{H}}$ with the scalar product and norm of $L^2$ and denote it by $\langle u , \v \rangle_{{\rm{H}}}$, $|u|_{{\rm{H}}}$ respectively for $u,\v \in {\rm{H}}$. We equip the space $\rm{V}$ with the scalar product and norm of $H^1$ and will denote it by $\langle \cdot , \cdot \rangle_V$ and $\|\cdot\|_{\rm{V}}$ respectively.

Let $\Pi : L^2 \rightarrow \rm{H}$ be Leray-Helmholtz projection operator which projects the vector fields on the plane of divergence free vector fields. We denote by $A: \rm{D}(A) \rightarrow {\rm{H}}$, the Stokes operator which is defined by
\begin{align*}
 \rm{D}(A) & = \rm{H} \cap H^2(\R^2),\\
 Au & = - \Pi\, \Delta u, ~~~u \in \rm{D}(A).
 \end{align*}
It is well known that $A$ is a self adjoint non-negative operator in ${\rm{H}}$. Note that $\Delta$ and $\Pi$ commute with each other. Moreover
\[ \rm{D}( (A+I)^{1/2}) = \rm{V} \quad \mbox{and} \quad \langle Au, u \rangle = |\nabla u|^2, \,\,\, u \in \rm{D}(A).\]
 From now onwards we will denote ${\rm{E}} := \rm{D}(A)$.

\subsection{Functional setting for a periodic domain}
\label{s:2.3}

We denote the bounded periodic domain by $\T$ which can be identified to a two dimensional torus. We introduce the following spaces:
\begin{equation}
\label{eq:functional on torus}
\begin{split}
\mathbb{L}^2_0 & = \{u \in L^2(\T, \R^2) : \int_{\T} u(x)\,dx = 0 \},\\
{\rm{H}} & = \{ u \in \mathbb{L}^2_0 : \nabla \cdot u = 0\},\\
\rm{V} &= H^1 \cap {\rm{H}}.
\end{split}
\end{equation}
We endow ${\rm{H}}$ with the scalar product and norm of $L^2$ and denote it by $\langle u , \v \rangle_{{\rm{H}}}$, $|u|_{{\rm{H}}}$ respectively for $u,\v \in {\rm{H}}$. We equip the space $\rm{V}$ with the scalar product $\langle \nabla u, \nabla \v \rangle_{{\rm{H}}}$ and norm $\|u\|_{\rm{V}}, u, \v \in \rm{V}$.\\

One can show that in the case of $\T$ $\rm{V}$-norm $\|\cdot \|_{\rm{V}}$, and $H^1$-norm $\|\cdot\|_{H^1}$ are equivalent on $\rm{V}$.\\

As before we denote by $A: \rm{D}(A) \rightarrow {\rm{H}}$, the Stokes operator which is defined by
\begin{align*}
\rm{D}(A) &= \rm{H} \cap H^2(\T),\\
 Au &= - \Delta u, ~~~u \in \rm{D}(A).
 \end{align*}
It is well known that $A$ is a self adjoint positive operator in ${\rm{H}}$. Moreover
\[\rm{D}(A^{1/2}) = \rm{V} \quad \mbox{and} \quad \langle Au , u \rangle = \|u\|^2_{\rm{V}} = |\nabla u|^2, \,\,\, u \in \rm{D}(A).\]

In the following subsection we will introduce a tri-linear form which is well defined for any general domain $\mathcal{O}$ and will state some of it's properties.

\subsection{Preliminaries}
\label{s:2.4}

From now onwards we denote our domain by $\mathcal{O}$ which can be either $\R^2$ or $\T$. We introduce a continuous tri-linear form $b : L^p \times W^{1,q} \times L^r \rightarrow \R$,
\[b(u,\v,w) = \sum_{i, j=1}^2 \int_{\dom} u^i \frac{\partial \v^j}{\partial x^i} w^j~dx,\]
where $p, q, r \in [1, \infty]$ satisfies
\[ \frac1p +\frac1q + \frac1r \leq 1.\]

We can define a bilinear map $B : {\rm{V}} \times {\rm{V}} \rightarrow {\rm{V}}^{\prime}$ such that
\[\langle B(u,\v), \phi \rangle = b(u,\v, \phi),~~~~ \text{for}~u, \v, \phi \in {\rm{V}},\]
where $\langle \cdot, \cdot \rangle$ denotes the duality between $V$ and $V^\prime$. If $u \in \rm{V}, \v \in \rm{E}$ and $\phi \in \rm{H}$ then

\[|b(u,\v,\phi)| \leq \sqrt{2}\, |u|^{\frac12}_{\rm{H}}\, \|u\|^{\frac12}_{\rm{V}}\,\|\v\|^{\frac12}_{\rm{V}}\, |\v|^{\frac12}_{\rm{E}}\,|\phi|_{\rm{H}}.\]

Thus $b$ can be uniquely extended to the tri-linear form (denoted by the same letter)
\[b: \rm{V} \times \rm{E} \times \rm{H} \to \R.\]
We can now also extend the operator $B$ uniquely to a bounded linear operator
\[B : \rm{V} \times \rm{E} \to \rm{H}.\]

The following properties of the tri-linear map $b$ and the bilinear map $B$ are very well established in \cite{[Temam79]} and \ref{a:1},
\begin{align*}
& b(u,u,u) = 0,~~~\quad \,\, \quad u \in \rm{V},\\
& b(u,w,w) = 0,~~~\quad  \quad u \in \rm{V}, w \in H^1,\\
& \langle B(u,u), Au\rangle_{\rm{H}} = 0, \quad u \in \rm{D}(A).
\end{align*}

The 2D Navier-Stokes equations are given as following:
\begin{align}
\label{eq:2.1}
\begin{cases}
\dfrac{\partial u(x,t)}{\partial t} - \nu \Delta u(x,t) + (u(x,t) \cdot \nabla ) u(x,t) + \nabla p(x,t)  = 0,\\
\nabla \cdot u(x,t) = 0,\\
u(x,0) = u_0(x),\\
\end{cases}
\end{align}
where $x \in \dom$ and $t \in [0,T]$ for every $T > 0$; $u \colon \dom \to \mathbb{R}^2$ and $p\colon \dom  \to \mathbb{R}$ are velocity and pressure of the fluid respectively. $\nu$ is the viscosity of the fluid (with no loss of generality, $\nu$ will be taken  equal to $1$ for the rest of the article, except in the Section~\ref{s:5}).\\

With all the notations as defined in the subsections \ref{s:2.1} and \ref{s:2.2}, the Navier-Stokes equation \eqref{eq:2.1} projected on divergence free vector field is given by
\begin{align}
\label{eq:2.2}
\begin{cases}
\dfrac{du}{dt} + Au + B(u,u) = 0,\\
u(0) = u_0.
\end{cases}
\end{align}

Let us denote the set of divergence free $\mathbb{R}^2$-valued functions with unit $L^2$ norm, as following
\[\mathcal{M} = \{ u \in {\rm{H}} : |u| = 1\}.\]
Then the tangent space at $u$ is defined as,
\[T_u \mathcal{M} = \{\v \in {{\rm{H}}} : \langle \v, u \rangle_{\rm{H}} = 0 \},~~~~u \in \mathcal{M}.\]

 We define a linear map $\pi_u : {{\rm{H}}} \rightarrow T_u \mathcal{M}$ by
\[ \pi_u(\v) = \v - \langle \v, u \rangle_{\rm{H}}\, u, \]
then $\pi_u$ is the orthogonal projection from ${{\rm{H}}}$ into $T_u \mathcal{M}$.\\

 Let $F(u) = Au + B(u,u)$ and $\hat{F}(u)$ be the projection of $F(u)$ on the tangent space $T_u\mathcal{M}$, then
\begin{align*}
\hat{F}(u) &= \pi_u(F(u))\\
&= F(u) - \langle F(u), u \rangle_{\rm{H}} \, u \\
&= Au + B(u,u) - \langle Au + B(u,u), u \rangle_{\rm{H}} \, u \\
&= Au - \langle Au, u \rangle_{\rm{H}} \, u + B(u,u) - \langle B(u,u), u \rangle_{\rm{H}} \, u \\
&= Au - |\nabla u|^2\,u + B(u,u).
\end{align*}
The last equality follows from the identity that $\langle B(u,u), u\rangle_{\rm{H}} = 0$.\\

\begin{remark}
\label{rem2.1}
Since $\langle B(u,u) , u\rangle_{\rm{H}} = 0$ and $ u \in \mathcal{M}$, $B(u,u) \in T_u \mathcal{M}$.
\end{remark}
\medskip
Thus by projecting NSEs \eqref{eq:2.2} on the manifold $\mathcal{M}$, we obtain our constrained Navier-Stokes equation which is given by
\begin{align}
\label{eq:2.3}
\begin{cases}
\dfrac{du}{dt} + Au - |\nabla u|^2\,u + B(u,u) = 0,\\
u(0) = u_0 \in \rm{V} \cap \mathcal{M}.
\end{cases}
\end{align}

\section{Local solution : Existence and Uniqueness}
\label{s:3}

In this section we will establish the existence of a local solution of the problem \eqref{eq:2.3} by using fixed point method. We obtain certain estimates for non-linear terms of \eqref{eq:2.3} using preliminaries from the previous section. After obtaining these estimates we construct a globally Lipschitz map. Some ideas in the Subsection \ref{s:3.1} are based on \cite{[BGT03]}. \\

We use the following well established \cite{[Temam79]} result to obtain the estimates.
\begin{lemma} \label{lemma3.1}
For any open set $\Omega \subset \mathbb{R}^2$  and every $\v \in H^1$, we have
\[|\v|_{\mathbb{L}^4(\Omega)} \leq 2^{1/4}|\v|_{\mathbb{L}^2(\Omega)}^{1/2}|\nabla \v|^{1/2}_{\mathbb{L}^2(\Omega)}, \quad \v \in H^1(\Omega). \]
\end{lemma}
\medskip

In what follows we assume that $\rm{E}, \rm{V}$ and $\rm{H}$ are spaces defined before in Section \ref{s:2}.

\begin{lemma} \label{lemma3.2}
Let $G_1 \colon \rm{V} \rightarrow {\rm{H}}$ be defined by
\[ G_1(u)  = |\nabla u|^2\, u, \quad u \in \rm{V}. \]
Then there exists $C > 0$ such that for $u_1, \2 \in \rm{V}$,
\begin{equation}
\label{eq:3.1}
|G_1(\1) - G_1(\2)|_{\rm{H}} \leq C \|\1 - \2 \|_{\rm{V}}\big[ \|\1\|_{\rm{V}} + \|\2\|_{\rm{V}} \big]^2.
\end{equation}
\end{lemma}

\begin{proof}
Let us consider $\1, \2 \in \rm{V}$, then
\begin{align*}
|G_1(\1) - G_1(\2)|_{\rm{H}} &= \big| |\nabla \1|^2\, u_1 - |\nabla \2|^2\, \2 \big|_{\rm{H}} \\
&= \left| |\nabla \1|^2 \,\1 - |\nabla \1|^2 \,\2 + |\nabla \1|^2\, \2 - |\nabla \2|^2\, \2 \right|_{\rm{H}} \\
&= \left| |\nabla \1|^2\,(\1 - \2) + (|\nabla \1|^2 - |\nabla \2|^2)\, \2 \right|_{\rm{H}} \\
&\leq |\nabla \1|^2\,|\1 - \2|_{\rm{H}} + \left[|\nabla \1| + |\nabla \2| \right]\, \left[|\nabla \1| - |\nabla \2| \right]\,|\2|_{\rm{H}} \\
&\leq |\nabla \1|^2\,|\1 - \2|_{\rm{H}} + \left[|\nabla \1| + |\nabla \2| \right]\,|\nabla(\1 - \2)|\, |\2|_{\rm{H}} \\
&\leq C \left[ |\nabla \1|^2\,\|\1 - \2\|_{\rm{V}} + \left[|\nabla \1| + |\nabla \2| \right]\, |\nabla(\1 - \2)|\, \|\2\|_{\rm{V}} \right] \\
&\leq C\|\1 - \2\|_{\rm{V}} \left[|\nabla \1|^2 + |\nabla \2|\,\|\2\|_{\rm{V}} + |\nabla \1|\, \|\2\|_{\rm{V}} \right],
\end{align*}
where we have repeatedly used the fact that $\rm{V}$ is continuously embedded in $\rm{H}$. Thus we obtain,
\begin{align*}
|G_1(\1) - G_1(\2)|_{\rm{H}} &\leq C \|\1 - \2\|_{\rm{V}} \left[\|\1\|_{\rm{V}} + \|\2\|_{\rm{V}} \right]^2.
\end{align*}
\end{proof}

\medskip

\begin{lemma} \label{lemma3.3}
Let $G_2 \colon {\rm{E}} \to \rm{H}$ be defined by
\[ G_2(u)  = B(u,u), \quad u \in \rm{E}. \]
Then there exists $\tilde{C} > 0$ such that for $\1, \2 \in \rm{E}$,
\begin{equation}
\label{eq:3.2}
|G_2(\1) - G_2(\2)|_{\rm{H}} \leq \tilde{C} \left[\|\1\|_{\rm{V}}^{1/2} |\1|_{{\rm{E}}}^{1/2} \|\1 - \2\|_{\rm{V}} + \|\2\|_{\rm{V}}\|\1 - \2\|_{\rm{V}}^{1/2} |\1 - \2|_{{\rm{E}}}^{1/2}  \right].
\end{equation}
\end{lemma}

\begin{proof}
Let us take $\1, \2 \in {\rm{E}}$, then
\begin{align*}
|G_2(\1) - G_2(\2)|_{\rm{H}} &= \left|B(\1,\1) - B(\2,\2)\right|_{\rm{H}} \\
&= \left|B(\1,\1) - B(\2,\1) + B(\2,\1) - B(\2,\2)\right|_{\rm{H}} \\
&= \left|B(\1 - \2,\1) + B(\2,\1 - \2)\right|_{\rm{H}} \\
&= \left|\Pi \left[(\1 - \2)\cdot \nabla\, \1 \right] + \Pi \left[\2 \cdot \nabla\,(\1 - \2) \right] \right|_{\rm{H}} \\
&\leq \left|(\1 - \2) \cdot \nabla\, \1 \right|_{\rm{H}} + \left| \2 \cdot \nabla\, (\1 - \2) \right|_{\rm{H}} \\
&\leq |\1 - \2|_{L^4(\mathcal{O})}|\nabla \1|_{L^4(\mathcal{O})} + |\2|_{L^4(\mathcal{O})}|\nabla(\1 - \2)|_{L^4(\mathcal{O})}.
\end{align*}
 Now using Lemma~\ref{lemma3.1} and the embedding of $\rm{V}$ in $\rm{H}$, we obtain,
\begin{align*}
|G_2(\1) - G_2(\2)|_{\rm{H}} &\leq \sqrt{2}\,|\1 - \2|_{\rm{H}}^{1/2}|\nabla(\1 - \2)|_{\rm{H}}^{1/2} |\nabla \1|_{\rm{H}}^{1/2} |\nabla^2 \1|_{\rm{H}}^{1/2} \\ &~~~~+ \sqrt{2}\,|\2|_{\rm{H}}^{1/2} |\nabla \2|_{\rm{H}}^{1/2} |\nabla (\1 - \2)|_{\rm{H}}^{1/2} |\nabla^2(\1 - \2)|_{\rm{H}}^{1/2} \\
&\leq \sqrt{2}C \Big[\|\1-\2\|_{\rm{V}}\|\1\|_{\rm{V}}^{1/2}|\1|_{{\rm{E}}}^{1/2}\\
&~~~~~~~~~~~~ + \|\2\|_{\rm{V}}\|\1 - \2\|_{\rm{V}}^{1/2}|\1 - \2|_{{\rm{E}}}^{1/2} \Big].
\end{align*}
 Thus we obtain the following inequality
\[ |G_2(\1) - G_2(\2)|_{\rm{H}} \leq \tilde{C}\left[\|\1\|_{\rm{V}}^{1/2} |\1|_{{\rm{E}}}^{1/2} \|\1 - \2\|_{\rm{V}} + \|\2\|_{\rm{V}} \|\1 - \2\|_{\rm{V}}^{1/2} |\1 - \2|_{{\rm{E}}}^{1/2}  \right].\]
\end{proof}
\medskip

\subsection{Construction of a globally Lipschitz map}
\label{s:3.1}

 Let $ \theta : \mathbb{R}_{+} \rightarrow [0, 1]$ be a $C_0^{\infty}$ non-increasing function such that
\[ \inf_{x \in \mathbb{R}_{+}} \theta'(x) \geq -1,~~ \theta(x) = 1~\text{iff}~ x \in [0,1]~\text{and}~ \theta(x) = 0~\text{iff}~ x \in [3, \infty) \]
and for $n \geq 1$ set $\theta_n(\cdot) = \theta(\frac{\cdot}{n})$. Observe that if $h : \mathbb{R}_{+} \rightarrow \mathbb{R}_{+}$ is a non-decreasing function, then for every $x, y \in \mathbb{R}_{+}$,
\[ \theta_n(x) h(x) \leq h(3n),~|\theta_n(x) - \theta_n(y)| \leq 3n|x - y|.\]
Set
\[ X_T = C([0,T]; \rm{V}) \cap L^2(0,T; {\rm{E}}), \]
with norm

\begin{align*}
|u|^2_{X_T} &= \sup _{t \in [0,T]} \|u(t)\|^2_{\rm{V}} + \int_0^T|u(t)|^2_{\rm{E}} dt.
\end{align*}

Let us define $G \colon  {\rm{E}} \rightarrow \rm{H}$ as
\begin{equation}
\label{eq:3.2a}
G(u) := G_1(u) - G_2(u) = |\nabla u|^2\,u - B(u,u)\\ .
\end{equation}

\begin{lemma}\label{lemma3.4}
Suppose $G \colon {\rm{E}} \rightarrow \rm{H}$ is a map defined in \eqref{eq:3.2a}. Let $T > 0$, define a map $ \Phi_{n,T}\colon X_T \rightarrow L^2(0,T; \rm{H})$ by
\begin{equation}
\label{eq:3.2b}
 \Phi_{n,T}(u)(x,t) = \theta_n(|u|_{X_t})G(u)(x,t).
 \end{equation}
Then $\Phi_{n,T}$ is globally Lipschitz and moreover, for any $u_1, u_2 \in X_T$,
\begin{equation}
\label{eq:3.3}
| \Phi_{n,T}(u_1) - \Phi_{n,T}(u_2)|_{L^2(0,T; {\rm{H}})} \leq {K(n,T)} |u_1 - u_2|_{X_T}T^{\frac14},
\end{equation}
where
\[K(n,T) = 3n \left(27 n^3 T^{1/4} + 9n^2 + 12n T^{1/4} + 2 \right),\]
depends on $n$ and $T$ only.
\end{lemma}

\begin{proof}
Assume that $u_1, u_2 \in X_T$. Set
\[ \tau_i = \inf\left\{ t \in [0, T] ; |u_i|_{X_t} \geq 3n \right \},\quad i = 1,2. \]
Without loss of generality assume that $\tau_1 \leq \tau_2$. Consider
\begin{align*}
|\Phi_{n,T}(\1) &- \Phi_{n,T}(\2)|_{L^2(0,T; \rm{H})}  = \left [ \int_0^T |\Phi_{n,T}(\1) - \Phi_{n,T}(\2)|_{\rm{H}}^2\,dt \right ]^{\frac12}\\
& = \left [ \int_0^T  \Big| \theta_n(|u_1|_{X_t}) G(\1) - \theta_n(|u_2|_{X_t}) G(\2)  \Big|_{\rm{H}}^2 dt \right ]^\frac12,
\end{align*}

 for $ i= 1,2$ $\theta_n(|u_i|_{X_t}) = 0$ for $ t \geq \tau_2$, thus we have
\begin{align*}
|\Phi_{n,T}(\1) - \Phi_{n,T}(\2)|_{L^2(0,T; \rm{H})} & = \left [ \int_0^{\tau_2}  \Big| \theta_n(|u_1|_{X_t}) G(\1) - \theta_n(|u_2|_{X_t}) G(\2)  \Big|_{\rm{H}}^2 dt \right ]^\frac12 \\
& = \Bigg [ \int_0^{\tau_2}  \Big| \theta_n(|u_1|_{X_t})\left[ G_1(\1) - G_2(\1) \right] \\
&~~~~~~~~~ - \theta_n(|u_2|_{X_t})\left[G_1(\2) - G_2(\2) \right]  \Big|_{\rm{H}}^2 dt \Bigg ]^\frac12 \\
&= \Bigg [ \int_0^{\tau_2}  \Big| \theta_n(|u_1|_{X_t}) G_1(\1) - \theta_n(|u_1|_{X_t}) G_1(\2) \\
&~~~~~~~~~ + \theta_n(|u_1|_{X_t}) G_1(\2) - \theta_n(|u_2|_{X_t}) G_1(\2) \\
&~~~~~~~~~ + \theta_n(|u_1|_{X_t}) G_2(\2) - \theta_n(|u_1|_{X_t}) G_2(\1) \\
&~~~~~~~~~ + \theta_n(|u_2|_{X_t}) G_2(\2) - \theta_n(|u_1|_{X_t}) G_2(\2) \Big|_{\rm{H}}^2 dt \Bigg ]^\frac12 .
\end{align*}

 Using the Minkowski inequality we get,
\begin{align*}
|\Phi_{n,T}(\1) - \Phi_{n,T}(\2)|_{L^2(0,T; \rm{H})} & \leq \left[ \int_0^{\tau_2} \Big| \theta_n(|\1|_{X_t}) \left[G_1(\1) - G_1(\2) \right] \Big|^2_{\rm{H}} dt \right]^{\frac12} \\
&~~~+ \left[ \int_0^{\tau_2} \Big| \left[ \theta_n(|\1|_{X_t}) - \theta_n(|\2|_{X_t}) \right] G_1(\2) \Big|^2_{\rm{H}} dt \right]^{\frac12} \\
&~~~+ \left[ \int_0^{\tau_2} \Big| \theta_n(|\1|_{X_t}) \left[G_2(\2) - G_2(\1) \right] \Big|^2_{\rm{H}} dt \right]^{\frac12} \\
&~~~+ \left[ \int_0^{\tau_2} \Big| \left[ \theta_n(|\2|_{X_t}) - \theta_n(|\1|_{X_t}) \right] G_2(\2) \Big|^2_{\rm{H}} dt \right]^{\frac12}.
\end{align*}

 Set
\begin{align*}
A_1 & = \left[ \int_0^{\tau_2} \Big| \left[ \theta_n(|\1|_{X_t}) - \theta_n(|\2|_{X_t}) \right] G_1(\2) \Big|^2_{\rm{H}} dt \right]^{\frac12}, \\
A_2 & = \left[ \int_0^{\tau_2} \Big| \theta_n(|\1|_{X_t}) \left[G_1(\1) - G_1(\2) \right] \Big|^2_{\rm{H}} dt \right]^{\frac12}, \\
A_3 & = \left[ \int_0^{\tau_2} \Big| \left[ \theta_n(|\2|_{X_t}) - \theta_n(|\1|_{X_t}) \right] G_2(\2) \Big|^2_{\rm{H}} dt \right]^{\frac12}, \\
A_4 & = \left[ \int_0^{\tau_2} \Big| \theta_n(|\1|_{X_t}) \left[G_2(\2) - G_2(\1) \right] \Big|^2_{\rm{H}} dt \right]^{\frac12}.
\end{align*}
 and hence
\begin{equation}
\label{eq:3.4}
|\Phi_{n,T}(\1) - \Phi_{n,T}(\2)|_{L^2(0,T; {\rm{H}})}  \leq A_1 + A_2 + A_3 + A_4.
\end{equation}

 Since $\theta_n$ is a Lipschitz function with Lipschitz constant $3n$ we obtain,
\begin{align*}
A_1^2 &= \int_0^{\tau_2} \left| \left[ \theta_n(|u_1|_{X_t}) -  \theta_n(|u_2|_{X_t}) \right] G_1(\2) \right|^2_{\rm{H}} dt\\
& \leq 9n^2 \int_0^{\tau_2} \left|\, |u_1|_{X_t} - |u_2|_{X_t}\right|^2_{\rm{H}} \left|G_1(\2) \right|_{\rm{H}}^2 dt .
\end{align*}

 Again using the Minkowski inequality we get
\begin{align}
A_1^2 &  \leq 9n^2 \int_0^{\tau_2} \left|u_1 - u_2 \right|_{X_t}^2 \left| G_1(\2) \right|_{\rm{H}}^2 dt \nonumber \\
\label{eq:3.5}
& \leq 9n^2 \left|u_1 - u_2 \right|_{X_T}^2 \int_0^{\tau_2} \left| G_1(\2) \right|_{\rm{H}}^2 dt .
\end{align}

 Now consider $ \int_0^{\tau_2} \left| G_1(\2) \right|_{\rm{H}}^2 dt$; using \eqref{eq:3.1} we get
\begin{align*}
\int_0^{\tau_2} \left| G_1(\2) \right|_{\rm{H}}^2 dt  & \leq C \int^{\tau_2}_0 {\|u_2(t)\|_{\rm{V}}^6 dt} \\
& \leq C^2 \left[ \sup_{t \in [0, \tau_2]} \|u_2(t)\|_{\rm{V}}^6 \right] \int_0^{\tau_2}{dt} \\
& \leq C^2 \left[\sup_{t \in [0, \tau_2]} \|u_2(t)\|_{\rm{V}}^2 \right]^3 \tau_2.
\end{align*}
Since
\[|u_2|_{X_{\tau_2}}^2 = \sup_{t \in [0,\tau_2]} \|u_2(t)\|_{\rm{V}}^2 + \int_0^{\tau_2}|u_2(t)|^2_{\rm{E}} dt, \]
thus \[ \sup_{t \in [0, \tau_2]} \|u_2(t)\|_{\rm{V}}^2 \leq |u_2|^2_{X_{\tau_2}},\]
and using
\[|u_2|_{X_{\tau_2}} \leq 3n, \]
we get
\begin{align*}
\int_0^{\tau_2} \left| G_1(\2) \right|_{\rm{H}}^2 dt & \leq C  \left[\sup_{t \in [0, \tau_2]} \|u_2(t)\|_{\rm{V}}^2 \right]^3 \tau_2 \\
& \leq C |u_2|^6_{X_{\tau_2}} \tau_2 \\
& \leq C(3n)^6 \tau_2.
\end{align*}
Hence, the inequality \eqref{eq:3.5} takes the form
\[A_1^2 \leq 9n^2 C\,|u_1 - u_2|^2_{X_T}(3n)^6\, \tau_2, \]
\begin{equation}
\label{eq:3.6}
A_1 \leq  (3n)^4 C\, |u_1 - u_2|_{X_T} \,\tau_2 ^{\frac12}.
\end{equation}
Similarly, since $\theta_n(|u_1|_{X_t}) = 0$  for  $t \geq \tau_1$ and $\tau_1 \leq \tau_2$, we have
\begin{align*}
A_2 &= \left[ \int_0^{\tau_2} \left| \theta_n(|u_1|_{X_t}) \left[ G_1(\1) - G_1(\2) \right]\, \right|^2_{\rm{H}} dt \right]^{\frac12}\\
& = \left[ \int_0^{\tau_1} \left| \theta_n(|u_1|_{X_t}) \left[ G_1(\1) - G_1(\2) \right]\, \right|^2_{\rm{H}} dt \right]^{\frac12}.
\end{align*}
Since $\theta_n(|u_1|_{X_t}) \leq 1$ for $t \in [0, \tau_1)$ and using \eqref{eq:3.1}, we have
\begin{align*}
A_2^2 & \leq  \int_0^{\tau_1} \left| G_1(\1) - G_1(\2) \right|^2_{\rm{H}} dt \\
& \leq C^2 \int_0^{\tau_1} \|\1 - \2\|_{\rm{V}}^2 \left[ \|\1\|_{\rm{V}} + \|\2\|_{\rm{V}} \right] ^4 dt \\
& \leq C^2 \sup_{t \in [0, \tau_1]}\|\1 - \2\|_{\rm{V}}^2 \int_0^{\tau_1} \left[ \|\1\|_{\rm{V}} + \|\2\|_{\rm{V}} \right] ^4 dt \\
& \leq C^2|\1 - \2|^2_{X_T} \sup_{t \in [0, \tau_1]}\left[ \|\1\|_{\rm{V}} + \|\2\|_{\rm{V}} \right] ^4 \int_0^{\tau_1} dt \\
& \leq C^2 |\1 - \2|^2_{X_T} \left[ |\1|_{X_{\tau_1}} + |\2|_{X_{\tau_1}} \right]^4 \tau_1 .
\end{align*}

 Since $|u_i|_{X_{\tau_i}} \leq 3n, i = 1,2.$ We get,
\begin{align*}
A_2^2 & \leq C^2|u_1 - u_2|^2_{X_T} \left[ |u_1|_{X_{\tau_1}} + |u_2|_{X_{\tau_1}} \right]^4 \tau_1 \nonumber \\
& \leq C^2 |u_1 - u_2|^2_{X_T} \tau_1 \left[ 3n + 3n \right]^4 \nonumber \\
A_2^2 &\leq  (6n)^4C^2 |u_1 - u_2|^2_{X_T} \tau_1.
\end{align*}

 Thus,
\begin{equation}
\label{eq:3.7}
A_2 \leq  (6n)^2C |u_1 - u_2|_{X_T} \tau_1^{\frac12}.
\end{equation}

 Now we consider,
\[A_3^2 = \int_0^{\tau_2} \Big| \left[ \theta_n(|\2|_{X_t}) - \theta_n(|\1|_{X_t}) \right] G_2(\2) \Big|^2_{\rm{H}} dt. \]

 Since $\theta_n$ is a Lipschitz function with Lipschitz constant $3n$ we obtain,
\begin{align*}
A_3^2 \leq 9n^2 \int_0^{\tau_2} \big| |u_2|_{X_t} - |u_1|_{X_t}\big|^2_{\rm{H}} \big|G_2(\2) \big|_{\rm{H}}^2 dt.
\end{align*}

 Using the Minkowski inequality we get
\begin{align}
A_3^2 &  \leq 9n^2 \int_0^{\tau_2} \big|u_1 - u_2 \big|_{X_t}^2 \big| G_2(\2) \big|_{\rm{H}}^2 dt \nonumber \\
\label{eq:3.8}
& \leq 9n^2 \big|u_1 - u_2 \big|_{X_T}^2 \int_0^{\tau_2} \big| G_2(\2) \big|_{\rm{H}}^2 dt.
\end{align}

 Now consider $ \int_0^{\tau_2} \big| G_2(\2) \big|_{\rm{H}}^2 dt$; using \eqref{eq:3.2} we get
\begin{align*}
\int_0^{\tau_2} \big| G_2(\2) \big|_{\rm{H}}^2 dt  & \leq \tilde{C}^{2} \int^{\tau_2}_0 {\|u_2(t)\|_{\rm{V}}^3 |\2|_{\rm{E}} dt} \\
& \leq \tilde{C}^2 \left[ \sup_{t \in [0, \tau_2]} \|u_2(t)\|_{\rm{V}}^2 \right]^{\frac32} \int_0^{\tau_2} |\2|_{\rm{E}} dt.
\end{align*}

 We apply the H\"older inequality to obtain,
\begin{align*}
\int_0^{\tau_2} \big| G_2(\2) \big|_{\rm{H}}^2 dt \leq \tilde{C}^2 |\2|_{X_{\tau_2}}^3 \left[ \int_0^{\tau_2} |\2|_{\rm{E}}^2 dt \right]^{\frac12} \left[ \int_0^{\tau_2} dt \right]^{\frac12} .
\end{align*}

 Now since $\int_0^{\tau_2} |\2|_{\rm{E}}^2 dt \leq |\2|^2_{X_{\tau_2}}$ and $|\2|_{X_{\tau_2}} \leq 3n$,
\begin{align*}
\int_0^{\tau_2} \big| G_2(\2) \big|_{\rm{H}}^2 dt & \leq \tilde{C}^2 |\2|_{X_{\tau_2}}^3 |\2|_{X_{\tau_2}} \tau_2^{\frac12} \\
& \leq \tilde{C}^2 (3n)^4 \tau_2^{\frac12}.
\end{align*}

 Hence, the inequality \eqref{eq:3.8} takes form
\[A_3^2 \leq 9n^2 \tilde{C}^2 |u_1 -u_2|^2_{X_T}(3n)^4 \tau_2^{\frac12} \]
\begin{equation}
\label{eq:3.9}
A_3 \leq (3n)^3 \tilde{C} |\1 - \2|_{X_T} \tau_2^{\frac14}.
\end{equation}

 Since $\theta_n(|\1|_{X_t}) = 0$ for $t > \tau_1$ and $\tau_1 < \tau_2$ we have,
\begin{align*}
A_4 & = \left[ \int_0^{\tau_2} \Big| \theta_n(|\1|_{X_t}) \left[G_2(\2) - G_2(\1) \right] \Big|^2_{\rm{H}} dt \right]^{\frac12} \\
& = \left[ \int_0^{\tau_1} \Big| \theta_n(|\1|_{X_t}) \left[G_2(\2) - G_2(\1) \right] \Big|^2_{\rm{H}} dt \right]^{\frac12}.
\end{align*}

 Since $\theta_n(|\1|_{X_{t}}) \leq 1$ for $t \in [0, \tau_1]$ and using \eqref{eq:3.2} we have,
\begin{align*}
A_4 & \leq \left[ \int_0^{\tau_1} \Big| G_2(\2) - G_2(\1) \Big|_{\rm{H}}^2 dt \right]^{\frac12} \\
& \leq \tilde{C} \left[ \int_0^{\tau_1} \left [ \|\1\|_{\rm{V}}^{1/2}|\1|_{\rm{E}}^{1/2}\|\1 - \2\|_{\rm{V}} + \|\1 - \2\|_{\rm{V}}^{1/2}|\1 - \2|_{\rm{E}}^{1/2} \|\2\|_{\rm{V}} \right] ^2 dt \right]^{\frac12}.
\end{align*}

 Now by the Minkowski inequality,
\begin{align*}
A_4 & \leq \tilde{C} \left[ \left[ \int_0^{\tau_1} |\1|_{\rm{E}}\|\1 - \2\|_{\rm{V}}^2\|\1\|_{\rm{V}} dt \right]^{\frac12} + \left[ \int_0^{\tau_1} \|\2\|_{\rm{V}}^2|\1 - \2|_{\rm{E}}^{1/2}\|\1 - \2\|_{\rm{V}}  dt \right]^{\frac12} \right]\\
& \leq \tilde{C} \left[ \sup_{t \in [0, \tau_1]} \|\1 - \2\|_{\rm{V}}^2 \left[\sup_{t \in [0, \tau_1}\|\1\|_{\rm{V}}^2 \right]^{\frac12} \int_0^{\tau_1} |\1|_{\rm{E}} dt \right]^{\frac12} \\
&~~~~~ + \tilde{C} \left[\sup_{t \in [0, \tau_1]}\|\2\|_{\rm{V}}^2 \left[ \sup_{t \in [0, \tau_1]}\|\1 - \2\|_{\rm{V}}^2 \right]^{\frac12} \int_0^{\tau_1} |\1 - \2|_{\rm{E}} dt \right]^{\frac12}.
\end{align*}

 Since
\[\sup_{t \in [0,\tau_1]} \|u_i\|_{\rm{V}}^2 \leq |u_i|^2_{X_{\tau_1}},~~~~~~\int_0^{\tau_1} |\1|_{\rm{E}}^2 dt \leq |u_1|^2_{X_{\tau_1}},~~~~~~|u_i|_{X_{\tau_1}} \leq 3n, \quad i = 1,2,\]
 and by using the H\"older inequality we obtain,
\begin{align*}
A_4 & \leq \tilde{C} \left[ |\1 - \2|^2_{X_T}|\1|_{X_{\tau_1}} \left[ \int_0^{\tau_1}|\1|_{\rm{E}}^2 dt \right]^{\frac12} \left[ \int_0^{\tau_1} dt \right]^{\frac12} \right]^{\frac12}\\
&~~~~~~+ \tilde{C} \left[ |\1 - \2|_{X_T}|\2|^2_{X_{\tau_1}} \left[ \int_0^{\tau_1}|\1 - \2|_{\rm{E}}^2 dt \right]^{\frac12} \left[ \int_0^{\tau_1} dt \right]^{\frac12} \right]^{\frac12}\\
&\leq \tilde{C} \left[ |\1 - \2|^2_{X_T}|\1|^2_{X_{\tau_1}} \tau_1^{\frac12} \right]^{\frac12} + \tilde{C} \left[ |\1 - \2|_{X_T}^2 |\2|^2_{X_{\tau_1}} \tau_1^{\frac12} \right]^{\frac12} \\
& \leq \tilde{C} |\1 - \2|_{X_T} \tau_1^{\frac14} \left[3n + 3n \right].
\end{align*}

 Thus
\begin{equation}
\label{eq:3.10}
A_4 \leq 6n \tilde{C} |\1 - \2|_{X_T} \tau_1^{\frac14}.
\end{equation}

 Now using \eqref{eq:3.6}, \eqref{eq:3.8}, \eqref{eq:3.9} and \eqref{eq:3.10} in \eqref{eq:3.4}, we obtain
\begin{align*}
|\Phi_{n,T}(\1) - \Phi_{n,T}(\2)|_{L^2(0,T; {\rm{H}})} & \leq (3n)^4 C |u_1 - u_2|_{X_T} \tau_2 ^{\frac12} + (6n)^2C |u_1 - u_2|_{X_T} \tau_1^{\frac12} \\ &~~~~ + (3n)^3 \tilde{C} |\1 - \2|_{X_T} \tau_2^{\frac14} + 6n \tilde{C} |\1 - \2|_{X_T} \tau_1^{\frac14}\\
& \leq (3n)^4 C |u_1 - u_2|_{X_T} T ^{\frac12} + (6n)^2C |u_1 - u_2|_{X_T} T^{\frac12} \\ &~~~~ + (3n)^3 \tilde{C} |\1 - \2|_{X_T} T^{\frac14} + 6n \tilde{C} |\1 - \2|_{X_T} T^{\frac14}\\
& = K(n,T)|\1 - \2|_{X_T} T^{\frac14},
\end{align*}
 where
\[K(n,T) = 3n \left(27 n^3 T^{1/4} + 9n^2 + 12n T^{1/4} + 2 \right),\]
is a constant which depends only on $n$ and $T$. Thus we have proved that $\Phi_{n,T}$ is a Lipschitz function and satisfies \eqref{eq:3.2}.
\end{proof}
\medskip

\subsection{Assumptions and definition of a solution}
\label{s:3.2}
 Assume that ${\rm{E}} \subset {\rm{V}} \subset {\rm{H}}$ continuously and $S(t)$ is a family of bounded linear operators on space ${\rm{H}}$ such that there exist $C_1, C_2 > 0$ s.t. \\
\begin{itemize}
\item[\textbf{A1.}] For every $T > 0$ and $f \in L^2(0,T; {\rm{H}})$ a function $u = S \ast f$, defined by
\[ u(t) = \int_0^T S(t - r)f(r) dr~~~~t \in [0, T], \]
belongs to $X_T$ and
\begin{equation}
\label{eq:3.11}
|u|_{X_T} \leq C_1 |f|_{L^2(0,T; {\rm{H}})}.
\end{equation}

\item[\textbf{A2.}] For every $T > 0$ and $u_0 \in \rm{V}$ a function $u = Su_0$ defined by
\[ u(t) = S(t) u_0, \]
belongs to $X_T$ and
\begin{equation}
\label{eq:3.12}
|u|_{X_T} \leq C_2\|u_0\|_{\rm{V}}.
\end{equation}
\end{itemize}
\smallskip

\begin{definition} \label{defn3.5}
\begin{itemize}
\item A solution of \eqref{eq:2.3} on $[0, T]$, $T \in [0, \infty)$ is a function $u \in X_T$ satisfying
\[ u(t) = S(t) u_0 + \int_0^t S(t-r) G(u(r)) dr~~~~\forall~t \in [0, T], \]
where $G : {\rm{E}} \rightarrow {\rm{H}}$ is defined by
\[ G(u) = |\nabla u|^2\, u - B(u, u), \quad u \in \rm{E}. \]

\item Let $\tau \in [0, \infty]$. A function $u \in C([0, \tau), \rm{V})$ is a solution to \eqref{eq:2.3} on $[0, \tau)$ iff $\forall ~T < \tau$, $u|_{[0, T]} \in X_T$ and satisfies
\[ u(t) = S(t) u_0 + \int_0^t S(t-r) G(u(r)) dr~~~~\forall~t \in [0, T]. \]
\end{itemize}
\end{definition}

 Define a function $\Psi_{n,T} : X_T \rightarrow X_T$ by
\[ \Psi_{n,T}(u) = S(t)u_0 + S \ast \Phi_{n,T}(u). \]

\begin{lemma}
\label{lemma3.6}
$u$ is the unique solution of \eqref{eq:2.3} iff $u$ is a fixed point of $\Psi_{n,T}$.
\end{lemma}
\medskip
\subsection{Local existence}
\label{s:3.3}

\begin{lemma} \label{lemma3.7} Assume that the assumptions (\textbf{A1})-(\textbf{A2}) hold. Consider a map $ \Psi_{n,T} : X_T \rightarrow X_T$ defined by
\[ \Psi_{n,T}(u) = Su_0 + S \ast \Phi_{n,T}(u), \]
where $\Phi_{n,T}$ is as in Lemma~\ref{lemma3.4}. Then there exists a constant $C_1 > 0$ such that $\Psi_{n,T}$ satisfies following inequality
\begin{equation}
\label{eq:3.13}
|\Psi_{n,T}(u_1) - \Psi_{n,T}(u_2)|_{X_T} \leq  C_1 K(n,T)|u_1 - u_2|_{X_T} T^{\frac14}, \quad \1, \2 \in X_T,
\end{equation}
where $K(n,T)$ has been introduced in Lemma \ref{lemma3.4}. Moreover, $\forall~\varepsilon \in (0,1)~\exists~ T_0 = T_0(n,\varepsilon)$ such that for every $u_0 \in \rm{V}$, $\Psi_{n,T}$ is an $\varepsilon$-contraction for $T \leq T_0$. \\
\end{lemma}

\begin{proof}
 The map $\Psi_{n,T}$ is evidently well defined. Now for any $\1, \2 \in X_T$,
\begin{align*}
|\Psi_{n,T}(\1) - \Psi_{n,T}(\2)|_{X_T} & = \Big|S(t)u_0 + S \ast \Phi_{n,T}(\1) - S(t) u_0 + S \ast \Phi_{n,T}(\2) \Big|_{X_T} \\
& = \Big|S \ast (\Phi_{n,T}(\1) - \Phi_{n,T}(\2)) \Big|_{X_T},
\end{align*}
 then by treating $S \ast (\Phi_{n,T}(\1) - \Phi_{n,T}(\2))$ as $ u $ and $\big[\Phi_{n,T}(\1) - \Phi_{n,T}(\2)\big] \in L^2(0,T; {\rm{H}})$ as $f$ in inequality \eqref{eq:3.11} and using Lemma~\ref{lemma3.4} we get
\begin{align*}
|\Psi_{n,T}(\1) - \Psi_{n,T}(\2)|_{X_T} & \leq C_1|\Phi_{n,T}(\1) - \Phi_{n,T}(\2)|_{L^2(0,T; {\rm{H}})} \\
& \leq C_1 K(n,T)|\1 - \2|_{X_T} T^{\frac14},
\end{align*}
 which shows that $\Psi_{n,T}$ is globally Lipschitz and satisfies \eqref{eq:3.13}.\\

 Let us fix $n\in \mathbb{N}$ and $\varepsilon \in (0,1)$. Since the constant $C_1$ is independent of $T$, we can find a $T_0 = T_0(n, \varepsilon)$ such that
\[C_1K(n,T_0)T_0^{\frac14} = \varepsilon,\]
 and thus $\Psi_{n,T}$ is an $\varepsilon$-contraction for $T \leq T_0$.\\
\end{proof}

 Let $\varepsilon \in (0,1)$ then from Lemma~\ref{lemma3.7}, $\Psi_{n,T}$ is an $\varepsilon$-contraction for $T = T_0(n, \varepsilon)$ and thus by Banach Fixed Point Theorem there exists a unique $u^n \in X_T$ \footnote{In fact $u^n$ should have been denoted by $u^{n,T}$ but we have refrained from this.} s.t.
\[u^n = \Psi_{n,T}(u^n).\]

 This implies that
\[u^n(t) = [\Psi_{n,T}(u^n)](t),~~~~~~t \in [0,T_0]. \]

 Let us define
\[\tau_n = \inf\{t \in [0,T_0] : |u^n|_{X_t} \geq n\}. \]
\medskip
\begin{remark}\label{rem3.8} If $|u^n|_{X_t} < n$ for each $t \in [0, T^n_0]$ then $\tau_n = T^n_0$.
\end{remark}
\medskip
\begin{theorem} \label{thm3.9}
 Let $R > 0$ be given then $\exists~T_{\ast} = T_{\ast}(R)$ such that for every $u_0 \in \rm{V}$ with $\|u_0\|_{\rm{V}} \leq R$ there exists a unique local solution $u : [0,T_{\ast}] \to \rm{V}$ of \eqref{eq:2.3}.
\end{theorem}

\begin{proof}
Let $R > 0$ and fix $\varepsilon \in (0,1)$. Let us choose\footnote{$\lfloor M \rfloor$ denotes the largest integer less than or equal to $M$.} $n = \lfloor \frac{C_2 R}{1-\varepsilon} \rfloor + 1$ where $C_2$ is as defined in \eqref{eq:3.12}. Now for these fixed $n$ and $\varepsilon$, $\exists~T_0(n,\varepsilon)$ such that $\Psi_{n,T}$ is an $\varepsilon$-contraction for all $T \leq T_0$. In particular, it is true for $T = T_0$ and hence by Banach Fixed Point Theorem $\exists!~u^n \in X_{T_0}$ such that
\[u^n = \Psi_{n,T}(u^n). \]
Note that we have
\begin{align*}
|u^n|_{X_{T_0}} & = |\Psi_{n,T}(u^n)|_{X_{T_0}} = |S u_0 + S \ast \Phi_{n,T}(u^n)|_{X_{T_0}} \\
& \leq |S u_0|_{X_{T_0}} + |S \ast \Phi_{n,T}(u^n)|_{X_{T_0}}.
\end{align*}
 Now from \eqref{eq:3.12} and Lemma~\ref{lemma3.7} we have,
\begin{align*}
|u^n|_{X_{T_0}}  \leq C_2\|u_0\|_{\rm{V}} + \varepsilon|u^n|_{X_{T_0}}.
\end{align*}
Hence
\[
(1-\varepsilon)|u^n|_{X_{T_0}} \leq C_2 R, \]
and so
\[
|u^n|_{X_{T_0}} \leq \frac{C_2 R}{1-\varepsilon} \le n. \]

 Now since $t \mapsto |\cdot|_{X_{t}}$ is an increasing function the following holds,
\[|u^n|_{X_t} \leq n ~~~~~\forall~ t \in [0, T_0].\]
In particular $|u^n|_{X_{T_0}} \le n$, i.e. $|u^n|_{X_{T_0}}$ is finite and thus $u^n \in X_{T_0}$.\\
This implies \[\theta_n(|u^n|_{X_t}) = 1, \quad t \in [0, T_0].\] Thus for $t \in [0, T_0]$,
\[u^n(t) = S(t) u_0 + \int_0^t S(t-r) G(u^n(r)) dr. \]
So $u^n$ on $[0, T_{\ast}(R)]$, where $T_{\ast} = T_0(n,\varepsilon)$, solves \eqref{eq:2.3} and $T_{\ast}$ depends only on $R$.

Thus we have proved the existence of a unique local solution of \eqref{eq:2.3} for every initial data $u_0 \in \rm{V}$, and this unique solution is denoted by $u$.
\end{proof}

\smallskip
\subsection{The solution stays on the manifold $\mathcal{M}$}
\label{s:3.4}

\begin{lemma} \label{lemma3.10} If $u$ is the solution of \eqref{eq:2.3} on $[0, \tau)$ then $u' \in L^2(0, T; {\rm{H}})$, for every $T < \tau$.
\end{lemma}

\begin{proof}
Let us fix $T < \tau$. Since $u$ is the solution of \eqref{eq:2.3} on $[0, \tau)$ it satisfies
\begin{equation}
\label{eq:3.14}
 \frac{du}{dt} = - Au + |\nabla u|^2\,u - B(u,u).
\end{equation}

 We will show that RHS of \eqref{eq:3.14} belongs to $L^2(0,T; {\rm{H}})$ and hence $u' \in  L^2(0,T; {\rm{H}})$. \\

 Since $u \in L^2(0, T; {\rm{E}})$, $Au \in L^2(0, T; {\rm{H}})$. From \eqref{eq:3.1} we have
\begin{align*}
\int_0^T \Big| |\nabla u(t)|^2\, u(t) \Big|_{\rm{H}}^2 dt & \leq \int_0^T C^2 \|u(t)\|_{\rm{V}}^6 dt \\
& \leq C^2 \sup_{t \in [0, T]}{\|u(t)\|_{\rm{V}}^6} \int_0^T dt \\
& \leq C^2 T \left[\sup_{t \in [0, T]}{\|u(t)\|_{\rm{V}}^2} \right]^3 \\
& \leq C^2 T |u|_{X_T}^6 < \infty,
\end{align*}
thus we have shown that $|\nabla u|^2\,u \in L^2(0,T; {\rm{H}})$.\\

 From \eqref{eq:3.2} we have,
\begin{align*}
\int_0^T \Big| B(u(t),u(t)) \Big|_{\rm{H}}^2 dt & \leq \tilde{C}^2 \int_0^T \|u(t)\|_{\rm{V}}^3|u(t)|_{\rm{E}} dt \\
& \leq \tilde{C}^2 \sup_{t \in [0,T]} {\|u(t)\|_{\rm{V}}^3} \int_0^T |u(t)|_{\rm{E}} dt \\
& \leq \tilde{C}^2 \left[ \sup_{t \in [0,T]} {\|u(t)\|_{\rm{V}}^2} \right]^{\frac32} \left[\int_0^T |u(t)|^2_{\rm{E}} dt \right]^{\frac12} \left[ \int_0^T dt \right]^{\frac12} \\
& \leq \tilde{C}^2 |u|_{X_T}^3 |u|_{X_T} T^{\frac12} < \infty.
\end{align*}
 Thus the non linear term from Navier-Stokes also belongs to $L^2(0,T; {\rm{H}})$ and hence RHS of \eqref{eq:3.14} belongs to $L^2(0,T; {\rm{H}})$ which implies $u' \in L^2(0,T; {\rm{H}})$ for all $T < \tau$.\\
\end{proof}

 The following Lemma is taken from \cite{[Temam79]}. It proves the existence of an absolute continuous function based on the regularity of the solution and it's time derivative.\\

\begin{lemma} \label{lemma3.11}
 Let ${\rm{V}}, {\rm{H}}$ and ${\rm{V}}'$ be the Gelfand triple. If a function $u \in L^2(0, T; {\rm{V}})$ and its weak derivative $u' \in L^2(0, T; {\rm{V}}')$ then $u$ is almost everywhere equal to a continuous function $\rm{v} : [0,T] \to \rm{H}$ such that
the function $[0,T] \ni t \mapsto |\rm{v}(t)|^2_{\rm{H}} \in \R$ is absolutely continuous and
\begin{equation}
\label{eq:3.15}
\frac12|\rm{v}(t)|_{\rm{H}}^2 = \frac12 |\rm{v}(0)|^2 + \int_0^t \langle u'(s), u(s) \rangle_{\rm{H}} ds, \,\,\, t \in [0, T].
\end{equation}
\end{lemma}

 \medskip

\begin{remark} \label{rem3.12} In the framework of Lemma~\ref{lemma3.11}, we can identify $\v$ with $u$ and so we get
\begin{equation}
\label{eq:3.17}
\frac12|u(t)|_{\rm{H}}^2 = \frac12 |u_0|^2 + \int_0^t \langle u'(s), u(s) \rangle_{\rm{H}} ds, \,\,\, t \in [0, \tau).
\end{equation}

Moreover, from Theorem~\ref{thm3.9} and Lemma~\ref{lemma3.10}
\begin{equation}
\label{eq:3.16}
\frac12\|u(t)\|_{\rm{V}}^2 = \frac12 \|u_0\|_{\rm{V}}^2 + \int_0^t \langle u'(s), u(s) \rangle_{\rm{V}} ds, \,\,\, t \in [0, \tau),
\end{equation}
where $\langle \cdot, \cdot \rangle_{\rm{V}}$ is defined in the Section~\ref{s:2} for $\R^2$ as well as $\T$.
\end{remark}

\medskip

\begin{theorem} \label{thm3.13} If $\tau \in [0, \infty]$, $u_0 \in \mathcal{M} \cap \rm{V}$ and $u$ is a solution to \eqref{eq:2.3} on $[0, \tau)$ then $u(t) \in \mathcal{M}$ for all $t \in [0, \tau)$.
\end{theorem}

\begin{proof} Let $u$ be the solution to \eqref{eq:2.3} and $u_0 \in \mathcal{M}\cap \rm{V}$. Let us define $\phi(t) = |u(t)|_{\rm{H}}^2 - 1$. Then $\phi$ is absolutely continuous and by Remark~\ref{rem3.12} and \eqref{eq:2.3} we have a.e. on $[0, \tau)$
\begin{align*}
\frac{d}{dt} \phi(t) =  \frac{d}{dt}[|u(t)|_{\rm{H}}^2 - 1] & = 2 \langle u'(t), u(t) \rangle_{\rm{H}} \\
& = 2 \langle -Au(t) + |\nabla u(t)|^2\, u(t) - B(u(t),u(t)), u(t) \rangle_{\rm{H}} \\
& = - 2 \langle Au(t), u(t)\rangle_{\rm{H}} + 2 |\nabla u(t)|^2 \langle u(t), u(t) \rangle_{\rm{H}} \\
& = -2 |\nabla u(t)|^2 + 2 |\nabla u(t) |^2 |u(t)|^2 \\
& = 2 |\nabla u(t)|^2(|u(t)|_{\rm{H}}^2 - 1)  = |\nabla u(t)|^2\,\phi(t).
\end{align*}
This on integration gives
\[\phi(t) = \phi(0)\exp{\left[\int_0^t |\nabla u(s)|^2 ds \right]},\,\,\, t \in [0, \tau). \]
Since $u_0 \in \mathcal{M}, \phi(0) = 0$ and also as $u \in X_T$ is the solution of \eqref{eq:2.3},
\[\int_0^t |\nabla u(s)|^2 ds \leq \int_0^t \|u(s)\|^2_{\rm{V}}\,ds  < \infty, \;\;\; t \in [0,\tau).\]
Hence we infer that $|u(t)|_{\rm{H}}^2 = 1$ for every $t \in [0, \tau)$. Thus $u(t) \in \mathcal{M}$ for every $t \in [0, \tau)$.
\end{proof}

\medskip

\begin{corollary}
\label{cor3.14}
Let the initial data $u_0 \in \mathcal{M}$ and $u$ is the solution to \eqref{eq:2.3} on $[0, \tau)$ then $u'(t)$ is orthogonal to $u(t)$ in ${\rm{H}}$ for every $t \in [0, \tau)$.
\end{corollary}

\begin{remark}
\label{rem3.15}
We can also prove Theorem~\ref{thm3.9} and Theorem~\ref{thm3.13} for any general bounded domain. Thus we can establish the existence of a local solution to \eqref{eq:2.3} for any general bounded domain and $\R^2$.
\end{remark}

\section{Global solution: Existence and Uniqueness}
\label{s:4}
In this section we will prove the existence of a global solution of \eqref{eq:2.3}. Lemma~A.\ref{lemmaA.1} and the Remark~\ref{rem4.1} play crucial role in proving the global existence of the solution. We use stitching argument to extend our solution from $[0, T], T < \infty$ on to the whole real line.

We recall the orthogonality property of the Stokes-operator in the following remark.

\begin{remark}
\label{rem4.1}
Note that one can show \cite{[Temam95]} that on a torus the following identity holds
\[ \langle B(u,u), Au \rangle_{\rm{H}} = 0,~~~~ \forall~u \in {\rm{V}}.\]
\end{remark}

 Let $u$ be the solution of \eqref{eq:2.3}. We define the energy of our system by
\[\mathcal{E}(u) = \frac{1}{2}|\nabla u|^2.\]
Then
\begin{align*}
\nabla_{\mathcal{M}} \mathcal{E}(u) & = \Pi_u(\nabla \mathcal{E}) \\
& = \Pi_u(Au)\\
& = Au - |\nabla u|^2\, u.
\end{align*}
Thus, for $u \in \mathcal{M}$
\begin{equation}
\label{eq:4.1}
|\nabla_{\mathcal{M}} \mathcal{E}(u)|_{\rm{H}}^2 = |u|^2_{\rm{E}} - |\nabla u|^4.
\end{equation}

\medskip

\begin{lemma} \label{lemma4.2} If $u$ is the local solution of \eqref{eq:2.3} on $[0, \tau)$, then
\[\sup_{s \in [0,\tau)}\|u(s)\|_{\rm{V}} \leq \|u_0\|_{\rm{V}}.\]
\end{lemma}

\begin{proof} Let $u$ be the solution of \eqref{eq:2.3}. Then, from \eqref{eq:2.3}, Remark~\ref{rem3.12} and Corollary~\ref{cor3.14}, for any $t \in [0, \tau)$ we have,
\begin{align*}
\frac12 \|u(t)\|_{\rm{V}}^2 & = \frac12 \|u_0\|_{\rm{V}}^2 + \int_0^t \langle u'(s), u(s) \rangle_{\rm{V}} ds\\
& = \frac12 \|u_0\|^2_{\rm{V}} + \int_0^t \langle u'(s), u(s)\rangle_{\rm{H}}\,ds + \int_0^t \langle u'(s), A u(s)\rangle_{\rm{H}}\,ds \\
& = \frac12 \|u_0\|_{\rm{V}}^2 + \int_0^t \langle -Au(s) + |\nabla u(s)|^2\, u(s) - B(u(s),u(s)), Au(s) \rangle_{\rm{H}} ds \\
& = \frac12 \|u_0\|^2_{\rm{V}} +\int_0^t \left[ - \langle Au(s),Au(s) \rangle_{\rm{H}} + |\nabla u(s)|^2 \langle u(s) , Au(s) \rangle_{\rm{H}}\right]~ds \\
&~~~~ - \int_0^t \langle B(u(s),u(s)), Au(s)\rangle_{\rm{H}} \,ds\\
& = \frac12 \|u_0\|^2_{\rm{V}} + \int_0^t \left[ -|u(s)|^2_{\rm{E}} + |\nabla u(s)|^4 \right]\,ds.
\end{align*}

 Now from Theorem~\ref{thm3.13} we know that $u(t) \in \mathcal{M}$ for every $t \in [0, \tau)$ and hence by using \eqref{eq:4.1} we obtain,
\[ \frac12 \|u(t)\|_{\rm{V}}^2 = \frac12 \|u_0\|_{\rm{V}}^2 - \int_0^t \Big|[\nabla_{ \mathcal{M}} \mathcal{E}(u)](s)\Big|_{\rm{H}}^2 ds,\]
and thus
\[ \frac12 \|u(t)\|_{\rm{V}}^2 + \int_0^t \Big|[\nabla_{ \mathcal{M}} \mathcal{E}(u)](s)\Big|_{\rm{H}}^2 ds = \frac12 \|u_0\|_{\rm{V}}^2. \]
 Hence we have shown that
\[ \|u(t)\|_{\rm{V}} \leq \|u_0\|_{\rm{V}}, \quad t \in [0, \tau). \]
\end{proof}

\begin{lemma} \label{lemma4.3}
Let $0 \leq a < b < c < \infty$ and $u \in X_{[a,b]}, \v \in X_{[b,c]}$, such that $u(b^{-}) = \v(b^{+})$. Then $z \in X_{[a,c)}$ where,
\begin{eqnarray}\nonumber
z(t)  =
\begin{cases}
 u(t),~~~~t \in [a,b),\\
 \v(t),~~~~t \in [b,c).
 \end{cases}
\end{eqnarray}
\end{lemma}

\begin{proof} Let us take $0 \leq a < b < c < \infty$ and $u \in X_{[a,b]}, \v \in X_{[b,c]}$, such that $u(b^{-}) = \v(b^{+})$. Then for any $0 \leq t_1 < t_2 < \infty$, using the definition of the norm $|\cdot|_{X_{[t_1, t_2]}}$, we have
\begin{align*}
|z|^2_{X_{[a,c]}} & = \sup_{t \in [a,c]} \|z(t)\|_{\rm{V}}^2 + \int_a^c |z(t)|^2_{{\rm{E}}} dt \\
& \leq \sup_{t \in [a,b]} \|z(t)\|_{\rm{V}}^2 + \sup_{t \in [b,c]} \|z(t)\|_{\rm{V}}^2 + \int_a^b |z(t)|^2_{{\rm{E}}} dt + \int_b^c |z(t)|^2_{{\rm{E}}} dt
\end{align*}
Now by the definition of $z$ we have,
\begin{align*}
|z|^2_{X_{[a,c]}} & \leq \sup_{t \in [a,b]} \|u(t)\|_{\rm{V}}^2 + \sup_{t \in [b,c]} \|\v(t)\|_{\rm{V}}^2 + \int_a^b |u(t)|^2_{{\rm{E}}} dt + \int_b^c |\v(t)|^2_{{\rm{E}}} dt \\
& = \sup_{t \in [a,b]} \|u(t)\|_{\rm{V}}^2 + \int_a^b |u(t)|^2_{{\rm{E}}} dt + \sup_{t \in [b,c]} \|\v(t)\|_{\rm{V}}^2 + \int_b^c |\v(t)|^2_{{\rm{E}}} dt\\
& = |u|^2_{X_{[a,b]}} + |\v|^2_{X_{[b,c]}} .
\end{align*}
Now since $u \in X_{[a,b]}$ and  $\v \in X_{[b,c]}$ we have $|z|_{X_{[a,c]}} < \infty$, and thus $z \in X_{[a,c]}$.
\end{proof}

 We will use the following lemma to prove our main result of existence of the global solution.\\

\begin{lemma} \label{lemma4.4}
Let $\tau$ be finite and the initial data $u_0 \in \rm{V} \cap \mathcal{M}$. If $u : [0, \tau] \rightarrow \rm{V}$ is the solution of \eqref{eq:2.3} on $[0, \tau]$ and $\v : [\tau, 2 \tau] \rightarrow \rm{V}$ is the solution of \eqref{eq:2.3} on $[\tau, 2\tau]$ such that $u(\tau^-) = \rm{v}(\tau^+)$, then $z: [0, 2\tau] \rightarrow \rm{V}$ defined as
\begin{eqnarray} \nonumber
z(t) =
\begin{cases}
u(t),~~~~t \in [0, \tau],\\
\v(t) , ~~~~t \in [\tau, 2 \tau],
\end{cases}
\end{eqnarray}
is the solution of \eqref{eq:2.3} on $[0, 2 \tau]$ and $z \in X_{[0, 2 \tau]}$.
\end{lemma}

\begin{proof} Since $u$ is the solution of \eqref{eq:2.3} on $[0, \tau]$ then $u \in X_{[0 ,\tau]}$ and similarly $\v \in X_{[\tau, 2 \tau]} := C([\tau, 2 \tau]; \rm{V}) \cap L^2(\tau, 2\tau; {\rm{E}})$. Thus by Lemma~\ref{lemma4.3} and the definition of $z$, $z \in X_{[0, 2 \tau]}$. Now we are left to show that $z:[0, 2 \tau] \rightarrow \rm{V}$ defined as
\begin{eqnarray} \nonumber
z(t) =
\begin{cases}
u(t),~~~~t \in [0, \tau],\\
\v(t) , ~~~~t \in [\tau, 2 \tau],
\end{cases}
\end{eqnarray}
is the solution of \eqref{eq:2.3} on $[0, 2 \tau]$. In order to achieve this we will have to show that $z$ satisfies \eqref{eq:4.2} for every $t \in [0, 2 \tau]$.
\begin{equation}
\label{eq:4.2}
 z(t) = S(t)z(0) + \int_0^t S(t-r)G(z(r)) dr.
 \end{equation}

 For $t \in [0, \tau), z$ satisfies \eqref{eq:4.2}, since $z(t) = u(t)$, $\forall ~ t \in [0, \tau]$ and $u$ is the solution of \eqref{eq:2.3} on $[0, \tau]$.\\
 For $t \in [\tau, 2 \tau]$, $z(t) = \v(t)$ and since $\v$ is the solution to \eqref{eq:2.3} on $[\tau, 2 \tau]$,

\[ z(t) = \v(t) = S(t- \tau)\v(\tau) + \int_{\tau}^t S(t-r) G(\v(r)) dr .\]
Now because of continuity of $u$ and $\v$, $\v(\tau) = u(\tau)$,
\begin{align*}
z(t) = S(t-\tau) \Big[S(\tau)u_0 + \int_0^{\tau} S(\tau - r) G(u(r))dr \Big] + \int_{\tau}^t S(t-r) G(\v(r))dr .
\end{align*}
Now using the definition of $z$ we obtain,
\begin{align*}
z(t)& = S(t)z(0) + \int_0^{\tau} S(t-r) G(z(r)) dr + \int_{\tau}^t S(t-r) G(z(r)) dr\\
& = S(t)z(0) + \int_0^t S(t-r)G(z(r)) dr .
\end{align*}
Thus $z$ satisfies \eqref{eq:4.2} on $[0, 2 \tau]$ and hence $z$ is a solution to \eqref{eq:2.3} on $[0, 2 \tau]$.\\
\end{proof}

\begin{proof}[Proof of Theorem~\ref{thm1.1}]
Let us take $u_0 \in \rm{V}$. Put $R = \|u_0\|_{\rm{V}}$. By Theorem~\ref{thm3.9} there exists a $T > 0 $ such that there exists a unique function $u: [0, T] \rightarrow \rm{V}$ which solves \eqref{eq:2.3} on $[0, T]$ and $u \in X_T$. Also by Lemma~\ref{lemma4.2} $\|u(T)\|_{\rm{V}} \leq R$ thus again by Theorem~\ref{thm3.9} there exists a unique function $\v: [T, 2T] \rightarrow \rm{V}$ which solves \eqref{eq:2.3} on $[T, 2T]$ and $\v \in X_{[T,2T]}$. Now if we define a new function $z: [0,2T] \rightarrow \rm{V}$ as
\begin{eqnarray} \nonumber
z(t) =
\begin{cases}
u(t),~~~~t \in [0, T],\\
\v(t) , ~~~~t \in [T, 2 T],
\end{cases}
\end{eqnarray}
then by Lemma~\ref{lemma4.4}, $z$ is also a solution of \eqref{eq:2.3} and $z \in X_{2T}$. Moreover $\|z(2T)\|_{\rm{V}} \leq R$. We can keep doing this and extend our solution further and hence obtaining a global solution of \eqref{eq:2.3} still denoted by $u$ such that $u \in X_T$ for every $T < \infty$. Each bit of the solution is unique on the respective domain and hence when we glue two unique bits we get a unique extension and thus obtain a unique global solution due to it's construction.
\end{proof}

\section{Convergence to the Euler equation}
\label{s:5}
In this section we are concerned with the convergence of the solution of the constrained Navier-Stokes equation, namely
\begin{align}
\label{eq:5.1}
\begin{cases}
\dfrac{du}{dt} + \nu Au - \nu\, |\nabla u|^2\, u + B(u,u) = 0,\\
u(0) = u_0^\nu \in \rm{V} \cap \mathcal{M},
\end{cases}
\end{align}
as $\nu$ vanishes on a torus. The curl of $u$ is defined as $\mathrm{Curl}(u):= D_1 u_2 - D_2 u_1$. We will prove Theorem~\ref{thm1.2} after several preliminary results.\\

\begin{remark}
\label{rem5.1}
$\mathrm{Curl}$ is a linear isomorphism between ${\rm{V}}$ and $L^2_0(\mathbb T^2)$, where \[L^2_0(\T):=\left\{\omega \in L^2(\mathbb T^2)\,:\: \int_{\mathbb T^2} \omega(x) dx=0\right\}.\]
Moreover for $u\in V$ and some universal constants $C>0$, $C_p>0$
\begin{equation}
\label{eq:ellreg1}
\|\Delta u\|_{L^2(\mathbb T^2)}\le C \|\nabla \curl{u}\|_{L^2(\mathbb T^2)},
\end{equation}
\begin{equation}
\label{eq:ellreg2}
\|\nabla u\|_{L^p(\mathbb T^2)}\le C_p \|\curl{u}\|_{L^\infty(\mathbb T^2)}.
\end{equation}
\end{remark}
This remark is proved in \ref{a:2}.\\

Hereafter $u^\nu$ is the solution to \eqref{eq:5.1}, and $\omega^\nu(t,x):=\curl{u^\nu(t)}(x)$. In particular, due to Remark~\ref{rem5.1} and Theorem~\ref{thm3.13}, $\omega^\nu\in C([0,T];L^2_0(\mathbb T^2)) \cap L^2(0, T; H^1(\T))$. It is then easy to check that $\omega^\nu$ is a weak solution to
\begin{align}
\label{eq:5.2}
\begin{cases}
\dfrac{d\omega^\nu}{dt} +\nabla \cdot (u^\nu\,\omega^\nu)= \nu \Delta \omega^\nu  + \nu\,  \|u^\nu\|_{\rm{V}}^2 \,\omega^\nu,\\
\omega^\nu(0) = \omega^\nu_0:=\curl{u_0^\nu} \in L^2_0(\mathbb T^2).
\end{cases}
\end{align}
\bigskip

\begin{proposition}
\label{prop5.2}
Let us fix $T > 0$, and assume that $\omega^\nu_0\in L^\infty(\mathbb T^2)$. Then
\begin{equation}
\label{eq:5.3}
\sup_{t\in [0,T]} |\omega^\nu(t)|_{L^\infty(\mathbb T^2)} \le
|\omega^\nu_0|_{L^\infty(\mathbb T^2)}
\exp\left(\nu \|u^\nu_0\|_{\rm{V}}^2\,T\right),
\end{equation}
\begin{equation}
\label{eq:5.3b}
\nu \int_0^T |\nabla \omega^\nu(t)|_{L^2(\mathbb T^2)}^2 dt \le
 \tfrac12 |\omega_0^\nu|_{L^2(\mathbb T^2)}^2+
\nu \,T \|u^\nu_0\|_{\rm{V}}^2 \,|\omega^\nu_0|_{L^\infty(\mathbb T^2)}^2
\exp\left(2 \nu \|u^\nu_0\|_{\rm{V}}^2\,T\right).
\end{equation}
\end{proposition}
\begin{proof}
Take $h\in C^2(\mathbb R)$, convex, with bounded second derivative. Then, since $\omega \in C([0,T];L^2_0(\mathbb T^2))$
\begin{align}
\label{eq:5.4}
\begin{split}
& \langle h(\omega^\nu(t)), \mathbf 1\rangle
-\langle h(\omega^\nu_0), \mathbf 1\rangle
\\ &
\qquad
 =\nu \int_0^t \left[ -\langle h''(\omega(s)),|\nabla \omega^\nu|^2(s)\rangle +\|u^\nu(s)\|_{\rm{V}}^2\,\langle h'(\omega^\nu(s)),\omega^\nu(s)\rangle \right]\,ds
\\ &
  \qquad
\le \nu \int_0^t \|u^\nu(s)\|_{\rm{V}}^2\,\langle h'(\omega^\nu(s)),\omega^\nu(s)\rangle\,ds.
\end{split}
\end{align}
For $p \geq 2, R>0$, take
\begin{equation}
h(w)\equiv h_{p,R}(w):=
\begin{cases}
|w|^p, & \text{if $|w|\le R$},
\\
R^p+p\,R^{p-1}(|w| - R) + \tfrac{p(p-1)}2 R^{p-2}(|w|-R)^2,
& \text{if $|w|>R$}.
\end{cases}
\end{equation}
Then $|h'(w) w|\le p\,h(w)$ and, since $\|u^\nu(s)\|_{\rm{V}}^2\le \|u^\nu_0\|_{\rm{V}}^2$
\begin{align}
\label{eq:5.5}
\begin{split}
&  \langle h(\omega^\nu(t)), \mathbf 1\rangle \le
\langle h(\omega^\nu_0), \mathbf 1\rangle + \nu\,p \int_0^t \|u^\nu_0\|_{\rm{V}}^2\,\langle h(\omega^\nu(s)),\mathbf 1\rangle\,ds.
\end{split}
\end{align}
By Gronwall inequality
\begin{align}
\label{eq:5.6}
\begin{split}
&  \langle h(\omega^\nu(t)), \mathbf 1\rangle \le
\langle h(\omega^\nu_0), \mathbf 1\rangle
\exp\left( \nu\,p \|u^\nu_0\|_{\rm{V}}^2 \,t\right), \quad \mbox{$t \in [0,T]$}.
\end{split}
\end{align}
Since
\begin{align}
\label{eq:5.7}
\begin{split}
|\omega^\nu|_{L^\infty} =\sup_{p,R}  \, \langle h_{p,R}(\omega^\nu), \mathbf 1\rangle^{1/p},
\end{split}
\end{align}
we get \eqref{eq:5.3}.

On the other hand, from the first equality in \eqref{eq:5.4}, taking now $h(w)=w^2/2$
\begin{equation*}
\begin{split}
\tfrac12 |\omega^\nu(T)|_{L^2(\mathbb T^2)}^2+\nu\!\!\int_0^T\!\!\! |\nabla \omega^\nu(t)|_{L^2(\mathbb T^2)}^2 dt
&  = \tfrac12 |\omega_0^\nu|_{L^2(\mathbb T^2)}^2+
\nu \!\! \int_0^T\!\!\! \|u^\nu(t)\|_{\rm{V}}^2\,|\omega^\nu(t)|_{L^2(\mathbb T^2)}^2
dt
\\ &
\le \tfrac12 |\omega_0^\nu|_{L^2(\mathbb T^2)}^2+
\nu \,T \|u^\nu_0\|_{\rm{V}}^2 \,|\omega^\nu_0|_{L^\infty(\mathbb T^2)}^2
e^{2 \nu T \|u^\nu_0\|_{\rm{V}}^2},
\end{split}
\end{equation*}
where in the last line we used \eqref{eq:5.3}. Hence \eqref{eq:5.3b}.
\end{proof}

\begin{proposition}
\label{prop5.3}
For each $\varphi \in {\rm{H}}^2(\mathbb T^2)$, and $\nu>0$
\begin{align}
\label{eq:5.8}
\begin{split}
\langle \omega^\nu(t)-\omega^\nu(s),\varphi\rangle
\le (t-s) \left(|\omega^\nu|_{L^\infty([0,T]\times \mathbb T^2)}+2\nu \|u_0^\nu\|_{\rm{V}}(1+\|u_0^\nu\|_{\rm{V}}^2) \right) |\varphi|_{H^2(\T)}.
\end{split}
\end{align}
\end{proposition}

\begin{proposition}
\label{prop5.4}
Suppose that, uniformly in $\nu$, $u^\nu_0$ is bounded in ${\rm{V}}$ and $\curl{u^\nu_0}$ is bounded in $L^\infty(\T)$. Then the sequence $u^\nu$ is precompact in $C([0,T];L^2(\T))$.
\end{proposition}
\begin{proof}
Let us take and fix $\varphi \in H^2(\T)$. Also fix $0 \leq s < t \leq T$. Then from the equation \eqref{eq:5.2} and $\|u^\nu(t)\|_{\rm{V}}^2 \le \|u_0^\nu\|_{\rm{V}}^2$
we get,
\begin{equation}
\label{eq:5.9}
\begin{split}
\left|\langle u^\nu(t)-u^\nu(s),\varphi\rangle\right|
\le
\nu \left|
\int_s^t \langle \Delta u^\nu,\varphi \rangle\,
 dr \right|+
  \nu  \|u_0^\nu\|_{\rm{V}}^2
 \int_s^t \left| \langle u^\nu,\varphi\rangle\right|\,dr
+\left| \int_s^t \langle u^\nu \nabla u^\nu,\varphi\rangle\,dr \right|.
\end{split}
\end{equation}
By \eqref{eq:ellreg1}, \eqref{eq:5.3b} and the hypotheses on the initial data, the first term in the r.h.s.\ is bounded by $C_T|\varphi|_{L^2}(t-s)^{1/2}$ for some constant $C_T$ independent on $\nu$. The second term in the r.h.s.\ of \eqref{eq:5.9} easily enjoys the same bound.  As for the third term in the r.h.s., for any $p>2$, $|u|_{L^\infty}\le C_p (|u|_{L^2}+ |\nabla u|_{L^p})$, so that from \eqref{eq:ellreg2} and \eqref{eq:5.3}, this term is still bounded by $C_T|\varphi|_{L^2}(t-s)^{1/2}$.

Therefore, since $u_0^\nu$ is bounded uniformly in $L^2(\mathbb T^2)$ by Poincar\'e inequality, it follows that $u^\nu$ is equibounded and equicontinuous in $L^2(\mathbb T^2)$ and, by Ascoli-Arzel\`a theorem, precompact in $C([0,T];L^2(\T))$.
\end{proof}

\begin{proof}[Proof of Theorem~\ref{thm1.2}]
Fix $T>0$. From Proposition~\ref{prop5.3}-\ref{prop5.4}, from each subsequence we can extract a further subsequence such that $\omega^\nu\to \omega$ in $C([0,T];H^{-2}(\T))$ and weakly in $L^\infty([0,T]\times \T)$, $u^\nu\to u$ weakly in $L^\infty([0,T];{\rm{V}})$ and in $C([0,T];L^2(\T))$. It is immediate to check that $\omega=\curl u$.

Notice that $\omega_0^\nu :=\curl{u_0^\nu}$ converges weakly in $L^\infty(\T)$ to $\omega_0:=\curl{u_0}$. Passing to the limit in the weak formulation of the equation one then has, for each $\varphi \in C^2([0,T]\times \T)$
\begin{equation}
\label{eq:5.11}
\langle \omega(t),\varphi(t)\rangle -\langle
\omega_0,\varphi(0)\rangle -\int_0^t \langle\omega(s),\partial_s \varphi(s)\rangle - \int_0^t \langle u\,\omega,\nabla \varphi\rangle=0,
\end{equation}
and $\omega(0)=\omega_0$. Recalling that $\omega=\curl u$
\begin{equation}
\label{eq:5.12}
\langle u(t),\nabla^\perp \varphi(t)\rangle -\langle u_0,\nabla^\perp \varphi(0)\rangle -\int_0^t
\langle u(s),\partial_s \nabla^\perp \varphi(s)\rangle - \int_0^t \langle u\cdot \nabla u,\nabla^\perp  \varphi\rangle=0 .
\end{equation}
Since $\langle u\,\omega,\nabla \varphi\rangle=\langle u\cdot \nabla u,\nabla^\perp  \varphi\rangle$ holds.

By Bardos uniqueness theorem \cite{[Bardos], [CW95]}, we conclude that $u^\nu\to u$.
\end{proof}

\pagebreak
\appendix

\section{Orthogonality of bilinear map to the Stokes operator}
\label{a:1}

\begin{lemmaA}
\label{lemmaA.1}
Let $x \in \mathcal{O}$, where $\mathcal{O} = \R^2 \text{ or } \T$ and $u \in \rm{D}(A)$,
then
\begin{equation}
\label{eq:a.1}
\langle B(u,u), Au \rangle_H = 0, \;\; \forall\, u \in \rm{D}(A).
\end{equation}
\end{lemmaA}

\begin{proof} The following proof has been modified \cite{[Temam95]} for $\R^2$.\\
Let $u \in \rm{D}(A)$ then, by the definition of $B(u,\v)$ and $Au$,
\begin{align*}
\langle B(u,u), Au \rangle_H & = \int_{\dom} (u(x) \cdot \nabla  ) u(x) \cdot Au(x)\, dx \\
& = \sum_{i,\,j,\,k\, = \,1}^2 \int_{\dom}(u_iD_iu_j)(- \Delta u_j)\, dx \\
& = - \sum_{i,\,j,\,k\, = \,1}^2 \int_{\dom} u_iD_iu_jD_k^2u_j\,dx.
\end{align*}
Now by integration by parts and the Stokes formula

\begin{align*}
\langle B(u,u), Au \rangle_H & = - \left(\sum_{i,\,j,\,k\, = \,1}^2 u_iD_iu_jD_ku_j\right)\Big\vert_{\partial \dom} + \sum_{i,\,j,\,k\, = \,1}^2 \int_{\dom} D_k(u_iD_iu_j)D_ku_j\, dx \\
& = \sum_{i,\,j,\,k\, = \,1}^2 \int_{\dom}D_ku_iD_iu_jD_ku_j\,dx + \sum_{i,\,j,\,k\, = \,1}^2\int_{\dom} u_iD_{k\,i}u_jD_ku_j \,dx.
\end{align*}
Now we will show that each of the terms in RHS will vanish.
We will consider the first term and show that it vanishes.
\begin{align*}
\sum_{i,\,j,\,k\, = \, 1}^2 D_ku_iD_iu_jD_ku_j & = (D_1u_1)^3 + D_1u_2D_2u_1D_1u_1 + D_1u_1(D_1u_2)^2 + (D_1u_2)^2D_2u_2 \\
&\,\,\, + (D_2u_1)^2 D_1u_1 + D_2u_2(D_2u_1)^2 + D_2u_1D_1u_2D_2u_2 + (D_2u_2)^3\\
& = (D_1u_1 + D_2u_2)\left[ (D_1u_1)^2 + (D_2u_2)^2 - D_1u_1D_2u_2\right] \\
&\,\,\, + D_1u_2D_2u_1(D_1u_1 + D_2u_2) + (D_1u_2)^2(D_1u_1 + D_2u_2) \\
&\,\,\, + (D_2u_1)^2(D_1u_1 + D_2u_2).
\end{align*}
Now since $\nabla \cdot u = D_1u_1 + D_2u_2 = 0$, the first term vanishes identically.\\
The second term vanishes because
\begin{align*}
2\sum_{i,\,j,\,k\, = \,1}^2 \int_{\dom} u_i D_{k\,i}u_jD_ku_j\, dx & = \sum_{i,\,j,\,k\, = \,1}^2 \int_{\dom}u_i D_i{(D_ku_j)^2} \,dx \\
& =  \left(\sum_{i,\,j,\,k\, = \,1}^2 u_i(D_ku_j)^2\right)\Big\vert_{\partial \dom} -  \sum_{i,\,j,\,k\, = \,1}^2 \int_{\dom} D_iu_i (D_ku_j)^2\,dx \\
& = -  \sum_{j,\,k\, = \,1}^2 \int_{\dom} (\nabla \cdot u) (D_ku_j)^2\,dx \;\; = 0.
\end{align*}
Thus we have shown that for every $u \in \rm{D}(A)$, $\langle B(u,u), Au \rangle_H = 0$.
\end{proof}

\section{Some results in the support of Section~\ref{s:5}}
\label{a:2}

\begin{remarkB}
\label{remarkB.1}
If $\nabla \cdot u=0$ and $\curl{u}=0$, then $u$ is constant by Hodge decomposition. In particular if $u\in \rm{V}$ and $\curl{u}=0$, then $u=0$.
\end{remarkB}

\begin{proof}[Proof of Remark~\ref{rem5.1}.]
We want to show that $\mathrm{Curl}$ is a linear isomorphism between $\rm{V}$ and $L^2_0(\T)$. It is clear that the map
\[\mathrm{Curl} : \rm{V} \ni u \mapsto \omega = \curl{u} \in L^2_0(\T),\]
is linear and continuous. Hence in order to prove the Remark~\ref{rem5.1} it is sufficient to find a continuous linear map
\begin{equation}
\label{eq:b.1}
 \Lambda : L^2_0(\T) \to \rm{V},
 \end{equation}
such that,
\begin{align}
\label{eq:b.2}
\rm{Curl} \circ \Lambda = \rm{id} \,\,  \mbox{on}\,\, L^2_0(\T),\\
\Lambda \circ \rm{Curl} = \rm{id} \,\, \mbox{on}\,\, \rm{V}.
\label{eq:b.3}
\end{align}
Let $\omega \in L^2_0(\T)$ then by elliptic regularity \cite{[GT01]} (applies also for $p \neq 2$) there exists a unique $\psi \in L^2_0(\T) \cap H^2(\T)$ such that
\begin{equation}
\label{eq:b.3a}
\Delta \psi = \omega,
\end{equation}
and the map
\[L_0^2 \ni \omega \mapsto \psi \in L_0^2 \cap H^2 ,\]
is bounded. Let us put $u = \nabla^\perp \psi$, i.e.
\begin{equation}
\label{eq:b.4}
u = (D_2 \psi, - D_1 \psi).
\end{equation}
Then $u \in H^1(\T)$ and $\nabla \cdot u = 0$ in the weak sense. Thus $u \in \rm{V}$. Using all of this we define the bounded linear map $\Lambda : L^2_0(\T) \ni \omega \mapsto u \in  \rm{V}.$ Now we are left to check that \eqref{eq:b.2} and \eqref{eq:b.3} holds for this $\Lambda$.\\

Let us take $\omega \in L^2_0(\T)$ and put $u := \Lambda(\omega) \in \rm{V}$. Now considering LHS of \eqref{eq:b.2},
\begin{align*}
(\rm{Curl} \circ \Lambda) (\omega) & = \rm{Curl}(u) = D_2 u_1 - D_1 u_2 \\
& = D_2D_2 \psi - (-D_1 D_1 \psi) = \Delta \psi = \omega,
\end{align*}
where we have used the definitions of $\psi$ and $u$ from \eqref{eq:b.3a} and \eqref{eq:b.4}. Hence we have established \eqref{eq:b.2}.\\

Now we take $\v \in \rm{V}$ and put $\omega = \rm{Curl}(\v) \in L^2_0(\T)$. Define $\psi \in L^2_0(\T) \cap H^2(\T)$ by
\begin{equation}
\label{eq:b.5}
\Delta  \psi = \omega.
\end{equation}
Observe that
\[\Delta \varphi = \rm{Curl}(D_2 \varphi, -D_1 \varphi), \quad \varphi \in H^2(\T).\]
Thus by \eqref{eq:b.5} and the definition of $u$ from \eqref{eq:b.4} we obtain
\[\rm{Curl}(u) = \rm{Curl}(\v),\]
where $u = \nabla ^ \perp \psi \in \rm{V}$.\\
Therefore using Remark~\ref{remarkB.1} $u = \v$, thus proving that $\rm{Curl}$ is a linear isomorphism between $\rm{V}$ and $L^2_0(\T)$. It is straightforward to show \eqref{eq:ellreg1}. Thus we are left to prove \eqref{eq:ellreg2}.\\

Let us fix $p \in (1, \infty)$ and  take $u \in H^{1,p}(\T)$. Denote $\omega = \curl{u} \in L^p_0(\T)$. From the first part of the proof there exists a bounded linear map $\Lambda : L^p_0(\T) \to H^{1,p}(\T)$
\[\Lambda  : L^p_0 \ni \omega \mapsto u \in H^{1,p},\]
such that
\[\rm{Curl} \circ \Lambda = \mbox{id on }L^p_0(\T).\]
In particular, there exists a $C_p' > 0$,
\[|\Lambda \omega |_{H^{1,p}(\T)} \leq C_p' |\omega|_{L^p(\T)}, \quad \omega \in L^p_0(\T).\]
Hence
\begin{equation}
\label{eq:b.7}
|\nabla \Lambda \omega|_{L^p(\T)} \leq C_p' | \omega|_{L^p(\T)}, \quad \omega \in L^p_0(\T).
\end{equation}
Taking now $u \in H^{1,p}(\T)$. Putting $\omega = \curl{u}$ so that $\Lambda \omega = u$ from \eqref{eq:b.7} we infer \eqref{eq:b.8},
\begin{equation}
\label{eq:b.8}
|\nabla u|_{L^p(\T)} \leq C_p |\omega|_{L^p(\T)}.
\end{equation}
Now since $|\omega|_{L^p(\T)} \leq |\omega|_{L^\infty(\T)}$ for every $p$, we can establish \eqref{eq:ellreg2}.
\end{proof}

\bigskip
\bibliographystyle{plain}

\end{document}